\documentclass[final]{siamltex}

\usepackage{makeidx}         
\usepackage{graphicx}        
\usepackage{color}                            
\usepackage{multicol}        
\usepackage{amsmath}
\usepackage{amssymb}
\usepackage{latexsym}
\usepackage[utf8]{inputenc}

\newtheorem{remark}{Remark}[section]
\newtheorem{example}{Example}[section]
\def\ds{\displaystyle}

\title{On Boundary Exact Controllability of One-Dimensional Wave Equations with Weak and Strong Interior Degeneration }

\author{Peter I. Kogut\thanks{Department of Differential Equations, Oles Honchar Dnipro
National University, Gagarin av., 72, 49010 Dnipro,
Ukraine (\tt p.kogut@i.ua)}  \and Olha P. Kupenko\thanks{Department of System Analysis and Control, National
Technical University "Dnipro Polytechnics", Yavornitsky av., 19, 49005 Dnipro, Ukraine;  Institute of Applied and System Analysis, Ihor Sikorsky National Technical University of Ukraine ``Kiev Polytechnical Institute'', Peremogy av., 37, build. 35, 03056 Kiev, Ukraine } ({\tt kupenko.olga@gmail.com}) \and G\"{u}nter Leugering\thanks{Department Mathematik
Lehrstuhl II Universit\"{a}t Erlangen-N\"{u}rnberg Cauerstr. 11
D-91058 Erlangen, Germany (\tt guenter.leugering@fau.de)}}
\begin{document}

\maketitle
\begin{abstract}
In this paper we study exact boundary controllability for a linear wave equation with strong and weak interior degeneration of the coefficient in the principle part of the elliptic operator. The objective is to provide a well-posedness analysis of the corresponding  system and derive conditions for its controllability through boundary actions.  Passing to a relaxed version of the original problem, we discuss existence and uniqueness of solutions, and using the HUM method we derive conditions on the rate of degeneracy for both exact boundary controllability and the lack thereof.
\end{abstract}

\begin{keywords}
Degenerate wave equation, boundary control, existence result, weighted Sobolev spaces, exact controllability.
\end{keywords}

\begin{AMS}
35L80, 49J20, 49J45, 93C73.
\end{AMS}

\pagestyle{myheadings}
\thispagestyle{plain}
\markboth{P.I.~Kogut, O.P.~Kupenko, G.~Leugering} {On Boundary Exact Controllability of Degenerate Wave Equations}

\section{Introduction}
\label{Sec_0}

In this paper we discuss exact boundary controllability  for  one-dimensional degenerate wave equations with a weak and strong  interior degeneration in the principle part of the elliptic operator. Let $[0,T]$ be a given time interval. For simplicity, let $c$ and $d$ be a given pair of real numbers such that $0\le c< 1< d\le 2$. We set
\[
\Omega_1=[c,1),\quad\Omega_2=(1,d],\quad\Omega=(c,d),\quad\text{and}\quad\Omega_0=\Omega\setminus\{1\}.
\]
 
Let $a:\overline{\Omega}\rightarrow\mathbb{R}$ be a given weight function with properties
\begin{enumerate}
	\item[(i)] $a(1)=0$, $a(x)>0$ for all $x\in\overline{\Omega}\setminus\{1\}=\Omega_1\cup\Omega_2$, and there exist subintervals $(x^\ast_1,1)\subset\Omega_1$ and $(1,x^{\ast}_2)\subset\Omega_2$ such that $a(\cdot)$ is monotonically decreasing on $(x^\ast_1,1)$,  monotonically increasing on $(1,x^{\ast}_2)$, and
	\begin{gather}
	\label{0.0a}
	\sup_{x\in [x^\ast_1,1)}\frac{(1-x)|a^\prime(x)|}{a(x)}=\lim_{x\nearrow 1} \frac{(1-x)|a^\prime(x)|}{a(x)}>0,\\
	\label{0.0b}
	\sup_{x\in (1,x^\ast_2]}\frac{(x-1)|a^\prime(x)|}{a(x)}=\lim_{x\searrow 1} \frac{(x-1)|a^\prime(x)|}{a(x)}>0;
	\end{gather}
	\item[(ii)] $a\in C(\overline{\Omega})\cap C^1(\overline{\Omega}\setminus\{1\})$;
	\item[(iii)] $\left(\sqrt{a}\right)_{x}\not\in L^\infty(\Omega)$ whereas $\left(\sqrt{a}\right)^{-1}_{x}\in L^\infty(\Omega)$.
\end{enumerate}

As an example of function $a:\overline{\Omega}\to \mathbb{R}_{+}$ with the above indicated properties (i)--(iii), we may consider the following one:
\begin{equation}
\label{1.0.01}
a(x)=\left\{
\begin{array}{ll}
(1-x)^{2p_1}, & \text{ if }\ x\in [c,1]\\
(x-1)^{2p_2}, & \text{ if }\ x\in (1,d],
\end{array}
\right. \quad\text{ with  $p_1,p_2>0$}.
\end{equation}
It is easy to check that, in this case, properties (i)--(iii) hold with $x_1^\ast=c$, $x_2^\ast=d$, and
\begin{gather*}
\sup_{x\in [x^\ast_1,1)}\frac{(1-x)|a^\prime(x)|}{a(x)}=\lim_{x\nearrow 1} \frac{(1-x)|a^\prime(x)|}{a(x)}=2p_1,\\
\sup_{x\in (1,x^\ast_2]}\frac{(x-1)|a^\prime(x)|}{a(x)}=\lim_{x\searrow 1} \frac{(x-1)|a^\prime(x)|}{a(x)}=2p_2,\\
\left(\sqrt{a}\right)_x\not\in L^\infty(\Omega)\ \text{ and }\ \left(\sqrt{a}\right)^{-1}_x\in L^\infty(\Omega)\ \text{ if $p_1$ and $p_2$ are less than $1$}.
\end{gather*}

We are concerned with the following controlled system
\begin{gather}
\label{0.1b}
y_{tt}-\left(a(x) y_x\right)_x =0\quad\text{in }\ (0,T)\times \Omega,\\
\label{0.1c}
y(t,c)=f_c(t),\quad
y(t,d)=f_d(t)\quad\text{on }\ (0,T),\\
\label{0.1d}
y(0,\cdot)=y_0,\quad y_t(0,\cdot)=y_1\quad\text{ in }\ \Omega,\\
\label{0.1e}
 f_c, f_d\in \mathcal{F}_{ad}=L^2(0,T).
\end{gather}
Here,  $y_0$, and $y_1$ are given functions, and $\mathcal{F}_{ad}$ stands for the class of admissible controls.

The system \eqref{0.1b}-\eqref{0.1e} describes the dynamics of a linear elastic string with out-of-the-plane displacement under the actions of boundary sources $f_c, f_d$ acting on the system as controls through the Dirichlet boundary conditions at $x=c$ and $x=d$. The coefficient $a(x)$ can be interpreted as the spatially varying stiffness (modulus of elasticity) of the elastic string. In contrast to the standard case that is widely studied in the literature (see, for instance, \cite{Russell}), where the stiffness is assumed to be positive and bounded away from zero, we assume that the string $[c,d]$ has a defect at the internal point $x_0=1$. In case the defect occurs at the endpoint $x_0=c$, the problem has been investigated in \cite{Cannarsa}. In the latter case, the spatial operator is related to a classical singular Sturm-Liouville-problem that has been treated already by Weyl in \cite{Weyl}. Degeneration in that context is related to the notions of limit-point and limit-cycle. In \cite{Cannarsa} the authors define 
\begin{equation}\label{mu}
\mu_a:= \sup\limits_{0<x\leq \ell}\frac{x|a'(x)|}{a(x)}
\end{equation}
for the problem on the interval $[0,\ell]$. The problem is called weakly damaged if $0\leq \mu_a< 1$ in which case $\frac{1}{a}\in L^1$, and strongly damaged in case $1\leq \mu_a< 2$. The authors show, among other things, one-sided boundary observability and consequently one-sided boundary exact controllability if $\mu_a<2$ with an observability/controllability time approaching $+\infty$ as $\mu_a\rightarrow 2$. Thus, for $\mu_a\geq 2$, these properties are lost.

Using the Liouville transform (see \cite{Eringen} 1954), it is possible to transform the system above into a homogeneous wave equation with singular potential on an interval that tends to infinity if $\mu_a\geq 2$. See e.g. \cite{Fardigola}, where controllability properties are investigated based on this transformation. Working in the $L^\infty$-framework, the author obtains similar results as in \cite{Cannarsa}.

The authors of this article are not aware of any publication where in-span degeneration of the wave equation is treated, in particular in the context of controllability or observability. For the parabolic case see \cite{Cannarsa2019}.
The main question that we are going to discuss in this paper, therefore,  is how the defect at the internal point $x_0=1$ affects the transmission conditions at the singular (damage) point and the corresponding solution of the system \eqref{0.1b}--\eqref{0.1e} as well as its observability or controllability properties. 

Such analysis could be important for many applications. In particular, for 
the cloaking problem (building of devices that lead to lack of obeservability) \cite{Green},
the evolution of damage in materials \cite{KL_2015}, optimization problems for elastic bodies arising in contact mechanics, coupled systems, composite materials \cite{KL}, where 'life-cycle-optimization' appears as a challenge.

The indicated type of degeneracy raises a number of new questions related to the well-posedness of the hyperbolic equations in suitable functional spaces as well as new estimates for their solutions. Hence, new tools are necessary for the analysis of the corresponding optimal control problems. It should be emphasized here that boundary value problems for degenerate elliptic and parabolic equations have received a lot of attention in the last years (see, for instance, \cite{Cannarsa2019,Canic,Colombo,MarMinRom,Rukav,Santambrogio}).
As for the control issue for degenerate wave equations, we already  mentioned \cite{Cannarsa,Gueye} (see also \cite{Lasiecka} for the sensitivity analysis of optimal control problems for wave equations in domain with defects).

The purpose of this paper is to provide a qualitative analysis of system \eqref{0.1b}--\eqref{0.1e}, prove an exact controllability result, and investigate how the degree of degeneracy in the principle coefficient $a(x)$ affects the system \eqref{0.1b}--\eqref{0.1e} and its solution.
In contrast to the recent results \cite{Cannarsa}, where the authors study controllability and observability for degenerate equation of the form \eqref{0.1b} with the degeneracy of \eqref{0.1b} at the boundary $x=c=0$, 
we focus on the case where the 'damaged' point is internal. So, our core idea is to to pass from the original initial-boundary value problem \eqref{0.1b}--\eqref{0.1e} to a relaxed version, namely, to some transmission problem with appropriate compatibility conditions at the 'damaged' point. We show that these conditions play a crucial role and essentially depend on the 'degree of degeneracy' (for some generalization we refer to \cite{KKL}). In multi-dimensional situation, when the 'damage zone' does not split the domain $\Omega$ onto several disconnected subdomains, there is an expectation that the boundary controllability of the corresponding hyperbolic equation with interior degeneration could be true under some suitable ’degree of degeneracy’ in the diffusion coefficients. Although, in general, the problem of boundary controllability in multi-dimensional case appears as a challenge. This issue will be considered in details in the forthcoming paper.  

In Section~\ref{Sec 1} we introduce a special class of weighted Sobolev spaces that are  associated with the original initial-boundary value problem. It allows us not only to to study in detail some properties of their elements in the regions which are in close vicinity to the 'damaged' point, but also propose an appropriate relaxation for to the initial-boundary value problems \eqref{0.1b}--\eqref{0.1d}. In Section~\ref{Sec 2}, we mainly focus on the well-posedness of the proposed relaxation for the original controlled system. It allows us to consider in Section~\ref{Sec_3} the issues related to the boundary observability of degenerate wave equations. In Section~\ref{Sec 4} we discuss the questions of exact and null boundary controllability of the original degenerate system and the lack of these properties for strong degeneration.

\section{Preliminaries}
\label{Sec 1}

To specify the original controlled system \eqref{0.1b}--\eqref{0.1e} and fix the main ideas, we begin with some preliminaries and assumptions. Let $a:\overline{\Omega}\rightarrow\mathbb{R}$ be a given weight function with properties (i)--(iii). For $u$ smooth, define the functional $\|\cdot\|_a$ as follows
\[
\|u\|_a=\left(\int_{\Omega}\left[u^2+au_x^2\right]\,dx\right)^{1/2}.
\]
Let $W^{1,2}(\Omega)$ be the standard Sobolev space. We denote by $H^1_a(\Omega)$, $H^1_{a,0}(\Omega)$, and $W^1_a(\Omega)$ the spaces which are defined as follows
\begin{itemize}
	\item $H^1_a(\Omega)$ is the closure of the set $\varphi\in C^\infty(\overline{\Omega})$ with respect to the $\|\cdot\|_a$-norm;
	\item $H^1_{a,0}(\Omega)$ is the closure of the set $C^\infty_c(\Omega)$ with respect to the $\|\cdot\|_a$-norm;
	\item $W^1_a(\Omega)$ is the space of maps $u\in L^2(\Omega)$ with distributional derivatives $u_x$ that satisfy $u_x\in L^2(\Omega,a\,dx)\cap L^1(\Omega)$, where 
	$$
	L^2(\Omega,a\,dx)=\left\{v:\Omega\rightarrow \mathbb{R}\ :\ v\ \text{ is measurable and }\ \int_{\Omega} a v^2\,dx<+\infty \right\}.
	$$ 
\end{itemize}

First note that since $C^\infty_c(\Omega)\subset C^\infty(\overline{\Omega})$, we have that $H^1_{a,0}(\Omega)\subset H^1_{a}(\Omega)$. Moreover, due to compactness of the embedding 
$$
H^1_{a}(\left(c,1-\varepsilon\right)\cup \left( 1+\varepsilon,d\right) )\hookrightarrow C^{0,1}([c,1-\varepsilon]\cup[1+\varepsilon,d]),\quad\text{ for all $\varepsilon>0$ small enough},
$$ 
we see that, if $y\in H^1_{a}(\Omega)$, then $y(\cdot)$ is an absolutely continuous function in $\overline{\Omega}\setminus\{1\}$. So, the conditions $y(c)=0$ and $y(d)=0$ are consistent for all $y\in H^1_{a,0}(\Omega)$. Therefore, $H^1_{a,0}(\Omega)$ can be equivalently defined as the closed subspace of $H^1_a(\Omega)$ such that	
\[
H^1_{a,0}(\Omega):=\left\{y\in H^1_a(\Omega)\ :\ y(c)=y(d)=0 \right\}.
\]
It is worth to notice that, unlike classical Sobolev space,  the subspace of smooth functions are not necessarily dense in $W^1_a(\Omega)$. So, for 'typical' weight functions $a:\overline{\Omega}\rightarrow\mathbb{R}$ with properties  (i)--(iii) it is unknown whether the identity $H^1_a(\Omega)= W^1_a(\Omega)$ is valid (for the corresponding examples we refer to \cite{Chiado_94,  Zh_94}). Therefore, it would be plausible to suppose that $H^1_a(\Omega)\subseteq W^1_a(\Omega)$, in general. 

We also have $H^1_{a,0}(\Omega)\subseteq W^1_{a,0}(\Omega)$, where
$W^1_{a,0}(\Omega)=\left\{y\in W^1_a(\Omega)\ :\ y(c)=y(d)=0 \right\}$.

Setting
\begin{gather}
\label{1.1a}
\mu_{1,a}:=\sup_{x\in [x^\ast_1,1)}\frac{(1-x)|a^\prime(x)|}{a(x)},\quad
\mu_{2,a}:=\sup_{x\in (1,x^\ast_2]}\frac{(x-1)|a^\prime(x)|}{a(x)},
\end{gather}
we deduce from properties (i)--(ii) that there exist constants $\kappa_{1,a}\ge 1$ and $\kappa_{2,a}\ge 1$ such that
\begin{equation}
\label{1.1b}
\sup_{x\in \Omega_1}\frac{(1-x)|a^\prime(x)|}{a(x)}=\kappa_{1,a}\mu_{1,a},\quad
\sup_{x\in \Omega_2}\frac{(x-1)|a^\prime(x)|}{a(x)}=\kappa_{2,a}\mu_{2,a}.
\end{equation}

Then the common characteristic of the weight functions $a:\overline{\Omega}\rightarrow\mathbb{R}$ with properties (i)--(iii) can be summarizing as follows (see see Theorems~3.1 and 3.2 in \cite{KKLW} for comparison).
\begin{proposition}
	\label{Prop 1.0}
	Let $a:\overline{\Omega}\rightarrow\mathbb{R}$ be a weight function satisfying properties (i)--(iii). Then 
	\begin{gather}
	\label{1.0.1b}
	a(x) \ge a(c)\frac{(1-x)^{\kappa_{1,a}\mu_{1,a}}}{(1-c)^{\kappa_{1,a}\mu_{1,a}}},\quad\forall\,x\in [c,1]\subset [0,1],\\
	\label{1.0.1c}
	a(x) \ge a(d)\frac{(x-1)^{\kappa_{2,a}\mu_{2,a}}}{(d-1)^{\kappa_{2,a}\mu_{2,a}}},\quad\forall\,x\in [1,d]\subset [1,2].
	\end{gather}
\end{proposition}
\begin{proof}
	Making use of representation \eqref{1.1b}, we get
	\begin{equation}
	\label{1.0.A1}
	(1-x)a^\prime(x)\stackrel{\text{by \eqref{1.1b}}}{\ge} -\kappa_{1,a}\mu_{1,a} a(x),\quad\forall\,x\in\Omega_1.
	\end{equation}
	Integrating this inequality over $[c,x]\subset [c,1)$, where $c>0$ by the initial assumptions,
	\[
	\int_c^x \frac{a^\prime}{a}\,ds\ge - \kappa_{1,a}\mu_{1,a}\int_c^x \frac{1}{1-s}\,ds,\quad\forall\, x\in[c,1),
	\]
	we arrive at the inequality \eqref{1.0.1b}.	Arguing in a similar manner, we have
	\begin{equation}
	\label{1.0.A2}
	(x-1)a^\prime(x)\stackrel{\text{by \eqref{1.1b}}}{\le}   \kappa_{2,a}\mu_{2,a}  a(x),\quad\forall\, x\in (1,d].
	\end{equation}
	Therefore,
	\[
	\int_x^d \frac{a^\prime}{a}\,ds\le \kappa_{2,a}\mu_{2,a} \int_x^d \frac{1}{s-1}\,ds,\quad\forall\, x\in (1,d].
	\]
	From this and the fact that $d<2$ we deduce \eqref{1.0.1c}.
\end{proof}

Arguing in a similar manner, it is easy to establish the following inequalities. 
\begin{gather}
\label{1.0.1d}
a(x) \ge a(x_1^\ast)\frac{(1-x)^{\mu_{1,a}}}{(1-x_1^\ast)^{\mu_{1,a}}},\quad\forall\,x\in [x_1^\ast,1]\subset [c,1],\\
\label{1.0.1e}
a(x) \ge a(x_2^\ast)\frac{(x-1)^{\mu_{2,a}}}{(x_2^\ast-1)^{\mu_{2,a}}},\quad\forall\,x\in [1,x_2^\ast]\subset [1,d].
\end{gather}

The next result is crucial for our further consideration and it explores some remarkable properties of the weight functions $a:\overline{\Omega}\rightarrow\mathbb{R}$ satisfying conditions (i)--(iii).

\begin{proposition}
	\label{Prop 1.0a}
	Let $a:\overline{\Omega}\rightarrow\mathbb{R}$ be a weight function satisfying properties (i)--(iii). Then the following assertions hold true:
	\begin{gather}
	\label{1.0.4a}
	0< \mu_{1,a}<2\quad\text{and}\quad 0< \mu_{2,a}<2.
	\end{gather}
\end{proposition}
\begin{proof}
Let $a:\overline{\Omega}\rightarrow\mathbb{R}$ be a given function with properties (i)--(iii). Setting
$$
k:=\left\|\left(\sqrt{a}\right)_x^{-1}\right\|^2_{L^\infty(\Omega)}\quad\text{ and }\quad\widehat{a}(x):=ka(x)\ \text{ for all $x\in\Omega$},
$$
we see that the function $\widehat{a}:\overline{\Omega}\rightarrow\mathbb{R}$ possesses all properties (i)--(iii) and
the direct calculations show that
$
\mu_{i,a}=\mu_{i,\widehat{a}}$ for $i=1,2$. 
Moreover, in this case, we have
\[
\left\|\left(\sqrt{\widehat{a}}\right)_x^{-1}\right\|_{L^\infty(\Omega)}=
\frac{1}{\sqrt{k}}\left\|\left(\sqrt{a}\right)_x^{-1}\right\|_{L^\infty(\Omega)}=1.
\]
Hence, without loss of generality, we can suppose that the function $a:\overline{\Omega}\rightarrow\mathbb{R}$ is such that
\begin{equation}
\label{2.21}
\left\|\left(\sqrt{a}\right)_x^{-1}\right\|_{L^\infty(\Omega)}\le 1.
\end{equation}
As a consequence of this condition, we have
\begin{equation}
\label{2.22}
\sup_{x\in [c,1)}\frac{2\sqrt{a(x)}}{|a^\prime(x)|}\le 1\quad\text{and}\quad
\sup_{x\in (1,d]}\frac{2\sqrt{a(x)}}{|a^\prime(x)|}\le 1.
\end{equation}
So, we can suppose that
\begin{equation}
\label{2.23}
2\sqrt{a(x)}\le |a^\prime(x)|,\quad\forall\,x\in[x^\ast_1,x^{\ast}_2].
\end{equation}
Since $a(\cdot)$ is a monotonically decreasing function on $(x^\ast_1,1)$ and $a(\cdot)$ is a monotonically increasing function on $(1,x^{\ast}_2)$, it follows from \eqref{2.23} that
\[
a^\prime(x)\le -2\sqrt{a(x)},\quad \forall\,x\in [x^\ast_1,1)\quad\text{and}\quad
a^\prime(x)\ge 2\sqrt{a(x)},\quad \forall\,x\in(1,x^{\ast}_2].
\]
Then, after integration, we obtain
\begin{gather}
\label{2.24a}
\int_x^{1} \frac{a^\prime(s)}{\sqrt{a(s)}}\,ds\le -2(1-x),\quad\forall\,x\in [x^\ast_1,1),\\
\label{2.24b}
\int_{1}^x \frac{a^\prime(s)}{\sqrt{a(s)}}\,ds\ge 2(x-1),\quad\forall\,x\in (1,x^{\ast}_2].
\end{gather}
Taking into account that $a(1)=0$, we deduce from \eqref{2.24a}--\eqref{2.24b} that
\begin{equation*}
\sqrt{a(x)}\ge 1-x,\quad\forall\,x\in [x^\ast_1,1)\quad\text{and}\quad
\sqrt{a(x)}\ge x-1,\quad\forall\,x\in (1,x^{\ast}_2],
\end{equation*}
and, as a consequence of these inequalities, we have
\begin{equation}
\label{2.25}
a(x)\ge (x-1)^2,\quad\forall\,x\in[x^\ast_1,x^{\ast}_2].
\end{equation}
Utilizing the monotonicity property of $a(\cdot)$ around the point $1$ and the fact that $\left(\sqrt{a}\right)_{x}\not\in L^\infty(\Omega)$, we deduce from \eqref{2.25} that there exists a positive value $\gamma\in (0,2)$ such that
\begin{equation}
\label{2.27}
a(x)=\mathcal{O}\left(|x-1|^{2-\gamma}\right)\quad\text{ in $[x^\ast_1,x^{\ast}_2]$},
\end{equation}
that is, $a(x)\thicksim |x-1|^{2-\gamma}$ near the degeneration point $1$.
Therefore, in view of representation \eqref{1.1a}, we finally have
\begin{align}
\label{2.26a}
\mu_{1,a}&:=\sup_{x\in [x^\ast_1,1)} \frac{|x-1||a^\prime(x)|}{a(x)}=2-\gamma<2,\\
\label{2.26b}
\mu_{2,a}&:=\sup_{x\in (1,x^\ast_2]} \frac{|x-1||a^\prime(x)|}{a(x)}=2-\gamma< 2.
\end{align}
\end{proof}

Let $V^1_{a,0}(\Omega)$ be some intermediate space with $H^1_{a,0}(\Omega)\subseteq V^1_{a,0}(\Omega)\subseteq W^1_{a,0}(\Omega)$. 
Our next intention is to show that due to the properties (i)--(iii) of the weight function $a:\overline{\Omega}\rightarrow\mathbb{R}$, $V^1_{a,0}(\Omega)$ is a Hilbert space with respect to the scalar product
\begin{equation}
\label{2.2}
\left<u,v\right>_{H^1_{a,0}(\Omega)}=\int_\Omega a(x) u^\prime(x) v^\prime(x)\,dx,\quad\forall\, u,v\in V^1_{a,0}(\Omega).
\end{equation}
To do so, it is enough to establish some version of Friedrichs's inequality. Following in many aspects \cite{Cannarsa, KKLW}, we will do it in two different manners (see Lemmas~\ref{Prop 1.0b} and \ref{Prop 1.0c} below)

\begin{lemma}
	\label{Prop 1.0b}
	Let $a:\overline{\Omega}\rightarrow\mathbb{R}$ be a given function with properties (i)--(iii). Let $V^1_{a,0}(\Omega)$ be some intermediate space with $H^1_{a,0}(\Omega)\subseteq V^1_{a,0}(\Omega)\subseteq W^1_{a,0}(\Omega)$. Let $u$ be an arbitrary element of  $V^1_{a,0}(\Omega)$. 
	Then the following inequality
	\begin{equation}
	\label{2.2.a}
	\|u\|_{L^2(\Omega)}\le D_a\|u^\prime\|_{L^2(\Omega,a\,dx)}=D_a\left(\int_{\Omega} a(x)u_x^2(x)\,dx\right)^{1/2}
	\end{equation}
	holds true for all $u\in V^1_{a,0}(\Omega)$, where  
	\begin{align}
	\label{2.4ab}
	D^2_a&=\max\{D_{1,a},D_{2,a}\},\\
	\label{2.4a}
	D_{1,a}&=\frac{\left(x_1^\ast-c\right)\left(2-x_1^\ast-c\right)}{2\min\limits_{x\in[c,x_1^\ast]} a(x) }+\frac{\left(1-x_1^\ast\right)^{2}}{a(x_1^\ast)\left(2-\mu_{1,a}\right)},\\
	\label{2.4b}
	D_{2,a}&=\frac{\left(d-x_2^\ast\right)\left(d+x_2^\ast-2\right)}{2\min\limits_{x\in[x_2^\ast,d]} a(x) }+\frac{\left(x_2^\ast-1\right)^{2}}{a(x_2^\ast)\left(2-\mu_{2,a}\right)}.
	\end{align}
\end{lemma}
\begin{proof}
	Let $u$ be a given element of $V^1_{a,0}(\Omega)$. Then, for any $x\in[c,1)$, we have
	\begin{align*}
	|u(x)|=\left|\int_c^x u^\prime(s)\,ds\right|&=\left|\int_c^x \sqrt{a(s)}u^\prime(s)\frac{1}{\sqrt{a(s)}}\,ds\right|
	\le
	\|u^\prime\|_{L^2(\Omega,a\,dx)}\left(\int_c^{x}\frac{ds}{a(s)}\right)^\frac{1}{2},
	\end{align*}
	where $L^2(\Omega,a\,dx)$ stands for the weighted Lebesgue space endowed with the norm\\ $\|f\|_{L^2(\Omega,a\,dx)}=\left(\int_{\Omega}f^2 a\,dx\right)^{1/2}$.
	
	From this and estimate \eqref{1.0.1d}, by Fubini's theorem, we obtain
	\begin{align}
	\notag
	\|&u\|^2_{L^2(\Omega_1)}\le \|u^\prime\|^2_{L^2(\Omega_1,a\,dx)} \int_c^{1} \int_c^{x}\frac{ds}{a(s)}\,dx=
	\|u^\prime\|^2_{L^2(\Omega_1,a\,dx)} \int_c^{1} \int_s^{1}\,dx\frac{ds}{a(s)}\\
	\notag
	&=\|u^\prime\|^2_{L^2(\Omega_1,a\,dx)} \left[\int_c^{x_1^\ast} \frac{1-s}{a(s)}\,ds+\int_{x_1^\ast}^{1} \frac{1-s}{a(s)}\,ds\right]
	\end{align}
	\begin{align}
	\notag
	&\stackrel{\text{by \eqref{1.0.1d}}}{\le}\|u^\prime\|^2_{L^2(\Omega_1,a\,dx)}\left[ 
	\frac{1}{\min\limits_{x\in[c,x_1^\ast]} a(x)}\int_c^{x_1^\ast} (1-s)\,ds+\frac{(1-x_1^\ast)^{\mu_{1,a}}}{a(x_1^\ast)}\int_{x_1^\ast}^{1}  (1-s)^{1-\mu_{1,a}}\,ds
	\right]\\
	&=\|u^\prime\|^2_{L^2(\Omega_1,a\,dx)}\left[\frac{\left(x_1^\ast-c\right)\left(2-x_1^\ast-c\right)}{2\min\limits_{x\in[c,x_1^\ast]} a(x) }+\frac{\left(1-x_1^\ast\right)^{2}}{a(x_1^\ast)\left(2-\mu_{1,a}\right)}\right].
	\label{2.5a}	
	\end{align}
	Arguing in a similar manner, for any $x\in (1,d]$, we have
	\[
	|u(x)|=\left|\int_x^d u^\prime(s)\,ds\right|\le \|u^\prime\|_{L^2(\Omega,a\,dx)}\left(\int_{x}^d\frac{ds}{a(s)}\right)^\frac{1}{2}.
	\]
	Then, estimate \eqref{1.0.1e} and Fubini's theorem lead us to the following chain of estimates
	\begin{align}
	\notag
	\|&u\|^2_{L^2(\Omega_2)}\le \|u^\prime\|^2_{L^2(\Omega_2,a\,dx)} \int_{1}^d \int_{x}^d\frac{ds}{a(s)}\,dx
	=
	\|u^\prime\|^2_{L^2(\Omega_2,a\,dx)} \int_{1}^d \int_{1}^s\,dx\frac{ds}{a(s)}\\
	\notag
	&=\|u^\prime\|^2_{L^2(\Omega_2,a\,dx)} \left[\int_{1}^{x_2^\ast} \frac{s-1}{a(s)}\,ds+\int_{x_2^\ast}^d \frac{s-1}{a(s)}\,ds\right]\\
	\notag
	&\stackrel{\text{by \eqref{1.0.1e}}}{\le}\|u^\prime\|^2_{L^2(\Omega_2,a\,dx)}\left[ 
	\frac{1}{\min\limits_{x\in[x_2^\ast,d]} a(x)}\int_{x_2^\ast}^d (s-1)\,ds+\frac{(x_2^\ast-1)^{\mu_{2,a}}}{a(x_2^\ast)}\int_1^{x_2^\ast}  (s-1)^{1-\mu_{2,a}}\,ds
	\right]\\
	\label{2.5b}
	&\le\|u^\prime\|^2_{L^2(\Omega_2,a\,dx)}\left[\frac{\left(d-x_2^\ast\right)\left(d+x_2^\ast-2\right)}{2\min\limits_{x\in[x_2^\ast,d]} a(x) }+\frac{\left(x_2^\ast-1\right)^{2}}{a(x_2^\ast)\left(2-\mu_{2,a}\right)}\right].
	\end{align}
	It remains to notice that due to Proposition~\ref{Prop 1.0a}, the obtained estimates \eqref{2.5a} and \eqref{2.5b} are consistent. 
\end{proof}

\begin{lemma}
	\label{Prop 1.0c}
	Let $a:\overline{\Omega}\rightarrow\mathbb{R}$ be a given function with properties (i)--(iii). Let $V^1_{a,0}(\Omega)$ be some intermediate space with $H^1_{a,0}(\Omega)\subseteq V^1_{a,0}(\Omega)\subseteq W^1_{a,0}(\Omega)$.  Then the estimate
	\begin{gather}
	\label{2.6a}
	\|u\|_{L^2(\Omega)}\le C_a\|u^\prime\|_{L^2(\Omega,a\,dx)}=C_a\left(\int_{\Omega} a(x)u_x^2(x)\,dx\right)^{1/2}
	\end{gather}
	holds true for all $u\in V^1_{a,0}(\Omega)$ with
	\begin{equation}
	\label{1.0.1dd}
	C_a^2=4\max\left\{\frac{ (1-c)^{\mu_{1,a}}}{\min\limits_{x\in[c,x_1^\ast]} a(x)}, \frac{(1-x_1^\ast)^{\mu_{1,a}}}{a(x_1^\ast)},\frac{ (d-1)^{\mu_{2,a}}}{\min\limits_{x\in[x_2^\ast,d]} a(x)}, \frac{(x_2^\ast-1)^{\mu_{2,a}}}{a(x_2^\ast)}\right\}
	\end{equation}
	In particular, if $x_1^\ast=c$ and $x_2^\ast=d$, then \eqref{2.6a} can be specified as follows
	\begin{gather}
	\label{2.6aa}
	\|u\|^2_{L^2(\Omega)}\le 4\max\left\{\frac{ (1-c)^{\mu_{1,a}}}{a(c)},\frac{ (d-1)^{\mu_{2,a}}}{a(d)}\right\}\|u^\prime\|^2_{L^2(\Omega,a\,dx)}
	\end{gather}
\end{lemma}
\begin{proof}
	Let $u$ be a given element of $V^1_{a,0}(\Omega)$. 
	Taking into account the following transformation
	\begin{align}
	\int_c^x (1-s)u^\prime(s)u(s)\,ds&=\frac{1}{2} \int_c^x (1-s)\frac{d}{ds} u^2(s)\,dx
	\label{2.7}
	=\frac{1}{2}(1-x)u^2(x)+\frac{1}{2}\int_0^xu^2(s)\,dx,
	\end{align}
	which is valid for all $x\in[c,1)$, we observe that
	\begin{align*}
	0&\le \int_c^x \left[(1-s)u^\prime(s)-\frac{1}{2}u(s)\right]^2\,ds\\
	&=\int_c^x \left[(1-s)^2\left[u^\prime(s)\right]^2+\frac{1}{4}u^2(s)-(1-s)u^\prime(s)u(s)\right]\,ds\\
	&\stackrel{\text{by \eqref{2.7}}}{=}
	\int_c^x \left[(1-s)^2\left[u^\prime(s)\right]^2-\frac{1}{4}u^2(s)\right]\,ds-\frac{1}{2}(1-x)u^2(x),\quad\forall\,x\in [c,1).
	\end{align*}
	From this, we deduce that
	\begin{align*}
	\int_c^x u^2(s)\,ds&\le 4 \int_c^x (1-s)^2\left[u^\prime(s)\right]^2\,ds\\
	&=4\int_c^{x_1^\ast} (1-s)^2\left[u^\prime(s)\right]^2\,ds+
	4\int_{x_1^\ast}^x (1-s)^2\left[u^\prime(s)\right]^2\,ds\\
	&\stackrel{\text{by \eqref{1.0.4a}}}{\le} 4 (1-c)^{\mu_{1,a}}\int_c^{x_1^\ast} \left[u^\prime(s)\right]^2\,ds+4\int_{x_1^\ast}^x (1-s)^{\mu_{1,a}}\left[u^\prime(s)\right]^2\,ds\\
	&\stackrel{\text{by \eqref{1.0.1d}}}{\le} \frac{4 (1-c)^{\mu_{1,a}}}{\min\limits_{x\in[c,x_1^\ast]} a(x)}\int_c^{x_1^\ast} a(s)\left[u^\prime(s)\right]^2\,ds+\frac{4(1-x_1^\ast)^{\mu_{1,a}}}{a(x_1^\ast)}\int_{x_1^\ast}^x a(s)\left[u^\prime(s)\right]^2\,ds\\	
	&\le \max\left\{\frac{4 (1-c)^{\mu_{1,a}}}{\min\limits_{x\in[c,x_1^\ast]} a(x)}, \frac{4(1-x_1^\ast)^{\mu_{1,a}}}{a(x_1^\ast)}\right\} \int_c^x a(s) \left[u^\prime(s)\right]^2\,ds,\quad\forall\,x\in[x_1^\ast,1).
	\end{align*}
	Taking the limit as $x\nearrow 1$ in the last relation, we arrive at the estimate
	\begin{equation}
	\label{2.8a}
	\|u\|^2_{L^2(\Omega_1)}\le \max\left\{\frac{4 (1-c)^{\mu_{1,a}}}{\min\limits_{x\in[c,x_1^\ast]} a(x)}, \frac{4(1-x_1^\ast)^{\mu_{1,a}}}{a(x_1^\ast)}\right\}\|u^\prime\|^2_{L^2(\Omega_1,a\,dx)}.
	\end{equation}
	
	By analogy with the previous case, we make use of the following transformation which is valid for each $x\in [1,d]$.
	\begin{equation}
	\int_x^d(s-1) u(s)u^\prime(s)\,ds=\frac{1}{2}\int_x^d (s-1) \frac{d}{ds}u^2(s)\,ds
	\label{2.9}
	=\frac{1}{2}\left[-(x-1)u^2(x)-\int_x^d u^2(s)\,ds\right].
	\end{equation}
	Then
	\begin{align*}
	0&\le \int_x^d \left[(s-1)u^\prime(s)+\frac{1}{2}u(s)\right]^2\,ds\\
	&=\int_x^d \left[(s-1)^2\left[u^\prime(s)\right]^2+\frac{1}{4}u^2(s)+(s-1)u(s)u^\prime(s)\right]\,ds\\
	&\stackrel{\text{by \eqref{2.9}}}{=}
	\int_x^d \left[(s-1)^2\left[u^\prime(s)\right]^2-\frac{1}{4}u^2(s)\right]\,ds- \frac{1}{2}(x-1)u^2(x).
	\end{align*}
	Since $(x-1)u(x)\ge 0$ for all $x\in [1,d]$, it follows that
	\begin{align*}
	\int_x^d &u^2(s)\,ds\le 4 \int_x^{x_2^\ast} (s-1)^2\left[u^\prime(s)\right]^2\,ds+4 \int_{x_2^\ast}^d (s-1)^2\left[u^\prime(s)\right]^2\,ds\\
	&\stackrel{\text{by \eqref{1.0.1e}, \eqref{1.0.4a}}}{\le} \left[ \frac{4(x_2^\ast-1)^{\mu_{2,a}}}{a(x_2^\ast)}\int_x^{x_2^\ast} a(s)\left[u^\prime(s)\right]^2\,ds+\frac{4 (d-1)^{\mu_{2,a}}}{\min\limits_{x\in[x_2^\ast,d]} a(x)}\int_{x_2^\ast}^d a(s)\left[u^\prime(s)\right]^2\,ds\right]\\
	&\le \max\left\{\frac{4 (d-1)^{\mu_{2,a}}}{\min\limits_{x\in[x_2^\ast,d]} a(x)}, \frac{4(x_2^\ast-1)^{\mu_{2,a}}}{a(x_2^\ast)}\right\}\int_x^d a(s) \left[u^\prime(s)\right]^2\,ds,\quad\forall\,x\in(1,x_2^\ast].
	\end{align*}
	As a result, passing to the limit in the last relation as $x\searrow 1$, we arrive at the inequality
	\begin{equation}
	\label{2.8b}
	\|u\|^2_{L^2(\Omega_2)}\le \max\left\{\frac{4 (d-1)^{\mu_{2,a}}}{\min\limits_{x\in[x_2^\ast,d]} a(x)}, \frac{4(x_2^\ast-1)^{\mu_{2,a}}}{a(x_2^\ast)}\right\}\|u^\prime\|^2_{L^2(\Omega_2,a\,dx)}.
	\end{equation}
	Thus, the announced estimate \eqref{2.6a} is a direct consequence of \eqref{2.8a} and \eqref{2.8b}.
\end{proof}

As a direct consequence of Lemmas~\ref{Prop 1.0b} and \ref{Prop 1.0c}, we have the following result.
\begin{theorem}[Friedrichs's inequality]
	\label{Th 2.0.1} Let
	$a:\overline{\Omega}\rightarrow\mathbb{R}$ be a given function with properties (i)--(iii). Let $u$ be an arbitrary element of some intermediate space  $V^1_{a,0}(\Omega)$. Then
	\begin{multline}
	\label{1.0.1ddd}
	\|u\|_{L^2(\Omega)}\le\min\{D_a,C_a\}\left(\int_{\Omega} a(x)u_x^2(x)\,dx\right)^{1/2}\\
	=\min\left\{\max\left\{\sqrt{D_{1,a}},\sqrt{D_{2,a}}\right\},C_a\right\}\left(\int_{\Omega} a(x)u_x^2(x)\,dx\right)^{1/2},\quad\forall\,u\in V^1_{a,0}(\Omega),
	\end{multline}
	where the constants $D_{1,a}$, $D_{2,a}$ and $C_a$ are defined in \eqref{2.4a}, \eqref{2.4b}, and \eqref{1.0.1dd}, respectively.
\end{theorem}

Now we can give the following conclusion.
\begin{corollary}
	Under the assumptions of Theorem~\ref{Th 2.0.1}, $V^1_{a,0}(\Omega)$ is a Hilbert space with respect to the scalar product
	\begin{equation}
	\label{1.0.0}
	\left<u,v\right>_{V^1_{a,0}(\Omega)}=\int_\Omega a(x) u^\prime(x) v^\prime(x)\,dx,\quad\forall\, u,v\in V^1_{a,0}(\Omega).
	\end{equation}
\end{corollary}

In what follows, we will distinguish two possible cases for the weight function $a:\overline{\Omega}\rightarrow\mathbb{R}$. Namely, we say that we deal with
\begin{itemize}
	\item a weak degeneration in \eqref{0.1b} if $a:\overline{\Omega}\rightarrow\mathbb{R}$ satisfies properties (i)--(iii) and $1/a\in L^1(\Omega)$;
	\item a strong degeneration in \eqref{0.1b} if $a:\overline{\Omega}\rightarrow\mathbb{R}$ satisfies properties (i)--(iii) and $1/a\not\in L^1(\Omega)$.
\end{itemize}

Starting with the weak degenerate case, we note that due to the continuous embedding $W^{1,1}(\Omega)\hookrightarrow C(\overline{\Omega})$ and the estimates
 
\begin{gather*}
\int_\Omega |y|\,dx\le |\Omega|^{1/2} \left(\int_\Omega |y|^2\,dx\right)^{1/2}\le \sqrt{|\Omega|}\, \|y\|_{a},\\
\int_\Omega| y_x|\,dx\le\left(\int_\Omega |y_x|^2 a\,dx\right)^{1/2}\left(\int_\Omega a^{-1}\,dx\right)^{1/2}\le C\|y\|_{a},
\end{gather*}
we have the following result (we refer to \cite[Proposition~2.5]{Cannarsa} for the details).
\begin{theorem}
\label{Th 1.2.000} Let $a:\overline{\Omega}\rightarrow\mathbb{R}$ be a weight function satisfying properties (i)--(iii) and $1/a\in L^1(\Omega)$. Then
$H^1_a(\Omega)=W^1_a(\Omega)$,
 $W^{1,2}(\Omega)\hookrightarrow W^1_a(\Omega)$, $W^1_a(\Omega)\hookrightarrow W^{1,1}(\Omega)$, $W^1_a(\Omega)\hookrightarrow\hookrightarrow L^1(\Omega)$ compactly, and $W^1_a(\Omega)$ is continuously embedded into the class of absolutely continuous functions on $\overline{\Omega}$, so
\begin{gather}
\label{1.0a}
\lim_{x\nearrow 1} y(x)=\lim_{x\searrow 1} y(x),\quad
|y(1)|<+\infty,
\quad\forall\,y\in H^1_a(\Omega),\\
\label{1.0aa}
\lim_{x\nearrow 1} \sqrt{a(x)} y(x)=\lim_{x\searrow 1} \sqrt{a(x)} y(x)=0,\quad\forall\,y\in H^1_{a}(\Omega).
\end{gather}
In addition, if $y$ is an arbitrary element of the space
\begin{equation}
\label{1.0ab}
W^2_a(\Omega):=\left\{ y\in W^1_a(\Omega)\ :\ a y_x\in W^{1,2}(\Omega)\right\},
\end{equation}
then the following transmission condition
\begin{equation}
\label{1.0b}
\lim_{x\nearrow 1} a(x) y_x(x)=\lim_{x\searrow 1} a(x) y_x(x)=L,\quad\text{with }\ |L|<+\infty,
\end{equation}
holds true.
\end{theorem}

However, the situation changes drastically if we deal with strong degeneration in \eqref{0.1b}. Indeed, let us consider the following example. Let $c=0$, $d=2$, and
\[
y(x)=\left\{
\begin{array}{ll}
|x-1|^{-\frac{1}{4}}-1, & \text{ if }\ x\in (0,1),\\
|x-1|^{\frac{1}{2}}, & \text{ if }\ x\in [1,2).
\end{array}
\right.
\]
Setting $a(x)=|x-1|^{7/4}$, we see that properties (i)--(iii) hold true. Moreover, in this case we have  $1/a\not\in L^1(\Omega)$. Then, in spite of the fact that the function $y:\Omega\rightarrow \mathbb{R}$ has a discontinuity of the second kind at $x^0=1$, a direct calculations show that $$ 
\|y\|^2_a=\int_0^2 \left[ y^2(x)+a(x) y_x^2(x)\right]\,dx<+\infty
$$
whereas
\[
a(x)y_x(x)=\left\{
\begin{array}{ll}
-\frac{1}{4}|x-1|^{\frac{1}{2}}, & \text{ if }\ x\in (0,1),\\
+\frac{1}{2}|x-1|^{\frac{5}{4}}, & \text{ if }\ x\in [1,2).
\end{array}
\right.
\]
So, transmission conditions \eqref{1.0b} for the given function $y$ with finite $H^1_a$-norm can be specified as follows
\begin{equation}
\label{1.0c}
\lim_{x\nearrow 1} a(x) y_x(x)=\lim_{x\searrow 1} a(x) y_x(x)=0.
\end{equation}

In fact, we have the following result (see  \cite[Proposition~2.5]{Cannarsa} for comparison).
\begin{theorem}
	\label{Th 1.1}
	Let $a:\overline{\Omega}\rightarrow\mathbb{R}$ be a weight function satisfying properties (i)--(iii) and $1/a\not\in L^1(\Omega)$. Let $V^1_{a}(\Omega)$ be some intermediate space with $H^1_{a}(\Omega)\subseteq V^1_{a}(\Omega)\subseteq W^1_{a}(\Omega)$. Then the following assertions hold true:
	\begin{gather}
	\label{1.1.1a}
	\lim_{x\nearrow 1} |x-1| y^2(x)=\lim_{x\searrow 1} |x-1| y^2(x)=0,\quad\forall\,y\in V^1_{a}(\Omega),\\
	\label{1.1.1b}
	\exists\, x_i\in\Omega_i,\ i=1,2, \text{ such that }\  y(x)=o\left(|x-1|^{-\frac{1}{2}}\right)\ \text{for a.a. }\ x\in (x_1,x_2),\\
	\label{1.1.1d}
	\lim_{x\nearrow 1} a(x) y_x(x)=\lim_{x\searrow 1} a(x) y_x(x)=0,\quad\forall\,y\in V^2_{a}(\Omega),\\
	\label{1.1.1e}
	\lim_{x\nearrow 1} |x-1|a(x) y_x(x)^2=\lim_{x\searrow 1} |x-1|a(x) y_x(x)^2=0,\quad\forall\,y\in V^2_{a}(\Omega),\\
    \label{1.1.1ee}
	\lim_{x\nearrow 1} a(x)\varphi_x(x) y(x)=\lim_{x\searrow 1} a(x) \varphi_x(x) y(x)=0,\quad\forall\,y\in V^1_{a}(\Omega),\ \forall\,\varphi\in V^2_{a}(\Omega),
	\end{gather}
	where the small symbol $o$  stands for the Bachmann-Landau asymptotic notation, and 
	\[
	V^2_a(\Omega):=\left\{ y\in V^1_a(\Omega)\ :\ a y_x\in W^{1,2}(\Omega)\right\}.
	\]
\end{theorem}
\begin{proof}
	Let $y\in V^1_a(\Omega)$. Without loss of generality, we assume (for the simplicity) that $x_1^\ast=c$ and $x_2^\ast=d$. Let us show that the function
	\[
	v(x)=\left\{
	\begin{array}{ll}
	(1-x)y^2(x), & c\le x<1,\\
	0, & x=1,\\
	(x-1)y^2(x), & 1<x\le d
	\end{array}
	\right.
	\]
	is continuous on $\overline{\Omega}$. Indeed, $v$ is locally absolutely continuous in $\overline{\Omega}\setminus\{1\}$ and
	\[
	v_x=\mathrm{sign}\,(x-1)\,y^2(x)+2|x-1|y(x)y_x(x),\quad\text{a.e. in }\ \overline{\Omega}.
	\]
	Since $y\in L^2(\Omega)$ and
	\begin{align}
	\notag
      \int_{\Omega}|x-1|^2 &y^2_x(x)\,dx\stackrel{\text{by \eqref{1.0.4a}}}{\le} \int_{\Omega_1}|x-1|^{\mu_{1,a}} y^2_x(x)\,dx
      +\int_{\Omega_2}|x-1|^{\mu_{2,a}} y^2_x(x)\,dx\\
      \notag
      &\stackrel{\text{by \eqref{1.0.1d}--\eqref{1.0.1e}}}{\le} \frac{(1-c)^{\mu_{1,a}}}{a(c)}\int_{\Omega_1} a(x)y^2_x(x)\,dx+
       \frac{(d-1)^{\mu_{2,a}}}{a(d)}\int_{\Omega_2} a(x)y^2_x(x)\,dx\\
       \label{1.1.2}
       &\le \max\left\{\frac{(1-c)^{\mu_{1,a}}}{a(c)},\frac{(d-1)^{\mu_{2,a}}}{a(d)}\right\}\|y\|^2_a,
	\end{align}
	it follows that $v_x\in L^1(\Omega)$. Hence, $v$ is an absolutely continuous functions and, as a consequence, the limits $\lim_{x\nearrow 1} |x-1| y^2(x)=\lim_{x\searrow 1} |x-1| y^2(x)=L$ do exist and must vanish, for otherwise $y^2(x)\sim L/|x-1|$ (near the point $x_0=1$) would be not integrable. So, we come into conflict with the initial condition: $y\in L^2(\Omega)$. From this and the fact that $a(x)=O(|x-1|)$ in some neighborhood of $x=1$,  we immediately deduce properties  \eqref{1.1.1a}--\eqref{1.1.1b}.
	
	To prove the equality \eqref{1.1.1d}, it is enough to observe that the function $a(x)y_x(x)$ with $y\in V^2_{a}(\Omega)$ is absolutely continuous. Hence, the limits $\lim_{x\nearrow 1} a(x) y_x(x)=\lim_{x\searrow 1} a(x) y_x(x)=L$ do exist and must vanish, for otherwise $a(x)y_x(x)^2\sim L^2/a(x)$ (near the point $x_0=1$) would be not integrable. So, we come into conflict with the initial condition: $y\in V^1_a(\Omega)$.
	
To establish property \eqref{1.1.1e}, we set
	\[
	v(x)=\left\{
	\begin{array}{ll}
	(1-x)a(x)y^2_x(x), & c\le x<1,\\
	0, & x=1,\\
	(x-1)a(x)y^2_x(x), & 1<x\le d,
	\end{array}
	\right.
	\]
	where $y$ is an arbitrary element of $V^2_{a}(\Omega)$. Then $v(x)$ is continuous on $\Omega$. Indeed, $v$ is locally absolutely continuous in $\overline{\Omega}\setminus\{1\}$ and
	\begin{align*}
	v_x(x)&=\mathrm{sign}\,(x-1)\,a(x) y^2_x(x)+2|x-1| y_x(x)\left(a(x)y_x(x)\right)_x\\
	&\quad-|x-1|a_x(x) y^2_x(x)=I_1(x)+I_2(x)+I_3(x),\quad\text{a.e. in }\ \overline{\Omega}.
	\end{align*}
	Since $y\in V^1_a(\Omega)$, it follows that $I_1\in L^1(\Omega)$. The same conclusion is true for the second term $I_2$. Indeed, in view of estimate \eqref{1.1.2}, we have
	\begin{align*}
	\|I_2\|_{L^1(\Omega)}&\le 2 \left(\int_{\Omega}|x-1|^2 y_x(x)^2\,dx\right)^\frac{1}{2}\left(\int_{\Omega} \left(a(x)y_x(x)\right)_x^2\,dx\right)^\frac{1}{2}\\
	&\stackrel{\text{by \eqref{1.1.2}}}{\le}2\sqrt{\max\left\{\frac{(1-c)^{\mu_{1,a}}}{a(c)},\frac{(d-1)^{\mu_{2,a}}}{a(d)}\right\}}\|y\|_a\|(ay_x)_x\|_{L^2(\Omega)}<+\infty.
	\end{align*}
	As for the third term, we see that
	\begin{align*}
	\|I_3\|_{L^1(\Omega)}&\stackrel{\text{by \eqref{1.0.1d}--\eqref{1.0.1e}}}{\le} { \mu_{1,a}}\int_{\Omega_1} a(x)y^2_x(x)\,dx
	+{\mu_{2,a}}\int_{\Omega_2} a(x)y^2_x(x)\,dx\\
	\notag
	&\stackrel{\text{by \eqref{1.0.4a}}}{\le} 2\int_{\Omega} a(x)y^2_x(x)\,dx\le 2\|y\|^2_a<+\infty.
	\end{align*}
	So, $v(x)$ is absolutely continuous in $\overline{\Omega}$. As a consequence, we see that
	the limits $\lim_{x\nearrow 1} |x-1|a(x) y_x(x)^2=\lim_{x\searrow 1} |x-1|a(x) y_x(x)^2=L$ do exist and must vanish, for otherwise $a(x)y_x(x)^2\sim L/|x-1|$ (near the point $x_0=1$) would be not integrable. 
	
It remains to prove relation \eqref{1.1.1ee}. We do it by proving that the function 
\[
	v(x)=\left\{
	\begin{array}{ll}
	a(x)\varphi_x(x)y(x), & c\le x<1,\\
	0, & x=1,\\
	a(x)\varphi_x(x)y(x), & 1<x\le d
	\end{array}
	\right.
	\]
	is continuous on $\overline{\Omega}$. This follows by the arguments as above, because
\begin{align*}
\|v_x\|_{L^1(\Omega)}&\le \int_\Omega |\sqrt{a}y_x||\sqrt{a}\varphi_x|\,dx+\int_\Omega |y||(a\varphi_x)_x|\,dx\\
&\le \|y\|_a\|\varphi\|_a+ \|y\|_{L^2(\Omega)}\|(a\varphi_x)_x\|_{L^2(\Omega)}<+\infty,
\end{align*}
and, therefore, $v$ is absolutely continuous in $\overline{\Omega}$. Thus, we see that
	the limits 
	\begin{equation}
	\label{A.1}
	\lim_{x\nearrow 1} a(x) \varphi_x(x) y(x)=\lim_{x\searrow 1} a(x) \varphi_x(x) y(x)=L
	\end{equation}
	do exist. To conclude the proof, we show that $L=0$. Indeed, in view of the property \eqref{1.1.1d}, we have
\begin{equation}
\label{A.2}
a(x)|\varphi_x(x)|=\left|\int_1^x \left(a\varphi_x\right)_x\,dx\right|\le \sqrt{|x-1|}\|(a\varphi_x)_x\|_{L^2(\Omega)},\ \forall\,x\in\Omega_0,\ \forall\,\varphi\in V^2_a(\Omega).
\end{equation}
Hence, if we assume that $L\ne 0$, then, in a small neighborhood $\mathcal{U}(1)$ of $x=1$, for any functions $y\in V^1_a(\Omega)$ and $\varphi\in V^2_a(\Omega)$, we have 
\begin{align*}
\frac{L}{2}\stackrel{\text{by \eqref{A.1}}}{\le} a(x)|\varphi_x(x)| |y(x)|&\stackrel{\text{by \eqref{A.2}}}{\le} \sqrt{|x-1|}|y(x)| \|(a\varphi_x)_x\|_{L^2(\Omega)}\\
& =\frac{\mathrm{Const}}{2}\,\sqrt{|x-1|}|y(x)|,\ \forall\,x\in \mathcal{U}(1).
\end{align*}
From this we deduce that
\[
\frac{L}{\sqrt{|x-1|}}\le \mathrm{Const}\, |y(x)|,\ \forall\,x\in \mathcal{U}(1).
\]
However, since $y$ is an $L^2(\Omega)$-function, this relations becomes inconsistent. Thus, $L=0$. 
\end{proof}

The main technical difficulty related to the problem \eqref{0.1b}--\eqref{0.1d} comes from the degeneration effect at the point $x_0=1$. Therefore, taking now into account Theorems~\ref{Th 1.2.000} and \ref{Th 1.1}, we specify the original initial-boundary value problem \eqref{0.1b}--\eqref{0.1d} in the form of the following transmission problem:
\begin{equation}
\label{1.1.3}
y_{tt}-\left(a(x) y_x\right)_x =0\quad\text{in }\ (0,T)\times (c,1)\ \text{and }\ (0,T)\times (1,d),
\end{equation}
with the initial conditions
\begin{equation}
\label{1.1.4}
y(0,\cdot)=y_0,\quad y_t(0,\cdot)=y_1\quad\text{ in }\ \Omega,
\end{equation}
the boundary conditions

\begin{equation}
\label{1.1.5}
y(t,c)=f_c(t),\quad
y(t,d)=f_d(t)\quad\text{on }\ (0,T),
\end{equation}

and  the transmission conditions:
\begin{enumerate}
\item[(I)] For the case $1/a\in L^1(\Omega)$
\begin{gather}
\label{1.1.6a}
\lim_{x\nearrow 1} y(t,x)=\lim_{x\searrow 1} y(t,x),\quad 0<t<T,\\
\label{1.1.6b}
\lim_{x\nearrow 1} a(x) y_x(t,x)=\lim_{x\searrow 1} a(x) y_x(t,x),\quad 0<t<T;
\end{gather}

\item[(II)] For the case $1/a\not\in L^1(\Omega)$
\begin{gather}
\label{1.1.7a}
\lim_{x\nearrow 1} a(x)\varphi_x(x) y(t,x)=0=\lim_{x\searrow 1} a(x) \varphi_x(x) y(t,x),\ \forall\,\varphi\in V^1_{a}(\Omega),\  0<t<T,\\
\label{1.1.7b}
\lim_{x\nearrow 1} a(x) y_x(t,x)=0=\lim_{x\searrow 1} a(x) y_x(t,x),\quad 0<t<T.
\end{gather}
\end{enumerate}

Since transmission conditions \eqref{1.1.6b}--\eqref{1.1.7b} were substantiated in Theorems~\ref{Th 1.2.000} and \ref{Th 1.1} if only $y(t,\cdot)\in V^2_a(\Omega)$ and $\varphi\in V^2_a(\Omega)$ (which mainly corresponds to the case of classical solutions), it is reasonable to consider the transmission problems \eqref{1.1.3}--\eqref{1.1.7b} as a relaxed version of the original problem \eqref{0.1b}--\eqref{0.1d}.

\begin{remark}
	\label{Rem 1.1.1}
	It is clear that the proposed relaxation is only a matter of regularity of some solutions. We refer to the recent papers \cite{BKL_2020,BK_2020}, where the authors consider a particular case of the problem \eqref{0.1b}--\eqref{0.1d} with $a(x)=\mathrm{const}\,|x-1|^\alpha$ for $\alpha\in[1,2)$, and they show that this problems is ill-posed and admits many solutions, but only one of them satisfies  transmission conditions \eqref{1.1.6a}--\eqref{1.1.6b} and has a continuously differentiable flux at $x=1$. As for the rest ones, they satisfy transmission conditions in the form \eqref{1.1.7a}--\eqref{1.1.7b}. So, the passage to the relaxed version does not change the original problem \eqref{0.1b}--\eqref{0.1d}.  
\end{remark}

\section{On well-posedness of the degenerate transmission problems}
\label{Sec 2}

In this section we recall the main results of semi-group theory concerning weak and classical notions of solutions for differential operator equation. 
Let $a:\overline{\Omega}\rightarrow\mathbb{R}$ be a given function with properties (i)--(iii). Let $V^1_{a,0}(\Omega)$ be some intermediate space with $H^1_{a,0}(\Omega)\subseteq V^1_{a,0}(\Omega)\subseteq W^1_{a,0}(\Omega)$. Such space can be constructed as the $\|\cdot\|_a$-closure of a linear span of $H^1_{a,0}(\Omega)$ with any element $y^\ast\in W^1_{a,0}(\Omega)\setminus H^1_{a,0}(\Omega)$. 
We associate with it  the Hilbert space $\mathcal{H}_{a}:=V^1_{a,0}(\Omega)\times L^2(\Omega)$ and endow it with the scalar product (see \cite{Cannarsa} for comparison)
\[
\left<\left[u\atop v\right],\left[\widetilde{u}\atop \widetilde{v}\right]\right>_{\mathcal{H}_a}=
\int_{\Omega} v(x)\widetilde{v}(x)\,dx+\int_{\Omega} a(x)u_x(x)\widetilde{u}_x(x)\,dx.
\]

We define the unbounded operator $\mathcal{A}:D(\mathcal{A})\subset \mathcal{H}_a\rightarrow \mathcal{H}_a$, associated with the problem \eqref{1.1.3}--\eqref{1.1.7b} provided $f_c(t), f_d(t)\equiv 0$, as follows
\begin{gather}
\label{2.1}
\mathcal{A}\left[u\atop v\right]=\left[v\atop \left(au_x\right)_x\right].
\end{gather}
and either
\begin{equation}
\label{2.1a}
D(\mathcal{A})=\left\{\left[u\atop v\right]\in W^2_a(\Omega)\times V^1_{a,0}(\Omega)\ :\ 
\begin{array}{c}
\ds\lim_{x\nearrow 1} u(x)=\lim_{x\searrow 1} u(x),\\
\ds\lim_{x\nearrow 1} a(x) u_x(x)=\lim_{x\searrow 1} a(x) u_x(x),\\
u(d)=0
\end{array}
\right\}
\end{equation}
if $1/a\in L^1(\Omega)$, or
\begin{equation}
\label{2.1b}
D(\mathcal{A})=\left\{\left[u\atop v\right]\in W^2_a(\Omega)\times V^1_{a,0}(\Omega)\ :\
\begin{array}{c}
\ds\lim_{x\nearrow 1} a\varphi_x u=0=\lim_{x\searrow 1} a \varphi_x u,\ \forall\,\varphi\in H^2_{a}(\Omega),\\
\ds\lim_{x\nearrow 1} a(x) u_x(x)=0=\lim_{x\searrow 1} a(x) u_x(x),\\
u(d)=0
\end{array}
\right\}
\end{equation}
provided $1/a\not\in L^1(\Omega)$.

Arguing as in \cite[Section II.2]{Gaevski}, it can be shown that $D(\mathcal{A})$ is a dense subset of $\mathcal{H}_a$.
\begin{lemma}
\label{Lemma 2.1.2}
$\mathcal{A}:D(\mathcal{A})\subset \mathcal{H}_a\rightarrow \mathcal{H}_a$ is the generator of a contraction semi-group in $\mathcal{H}_a$.
\end{lemma}
\begin{proof}
It is well-known that if $H$ is a Hilbert space and $B:D(B)\subset H\rightarrow H$ is a densely defined linear operator such that both $B$ and $B^\ast$ are dissipative, i.e.,
\[
\left<Bu,u\right>_{H}\le 0\quad\text{and}\quad \left<u,B^\ast u\right>_{H}\le 0\quad\forall\,u\in D(B),
\]
then $B$ generates a strongly continuous semi-group of contraction operators \cite[p.~686]{Lumer}. Let us show that the operator $\mathcal{A} : D(\mathcal{A})\subset\mathcal{H}_a \rightarrow \mathcal{H}_a$ satisfies the above mentioned properties.

Indeed, since the inclusion $\mathcal{A}\left[u\atop v\right]\in \mathcal{H}_a$ is obvious for all $\left[u\atop v\right]\in D(\mathcal{A})$, it remains to check whether the properties
\begin{equation}
\label{2.1.1}
\left<\mathcal{A}\left[u\atop v\right],\left[u\atop v\right]\right>_{\mathcal{H}_a}\le 0,\quad\text{and}\quad
\left<\left[u\atop v\right],\left(\mathcal{A}\right)^\ast\left[u\atop v\right]\right>_{\mathcal{H}_a}\le 0
\quad\forall\, \left[u\atop v\right]\in D(\mathcal{A})
\end{equation}
hold true. 
We do it for the case (II), $1/a\not\in L^1(\Omega)$, because the case (I) can be considered in a similar manner. The first inequality in \eqref{2.1.1} immediately follows from the definition of the set $D(\mathcal{A})$ and the following relations
\begin{multline}
\label{2.1.2}
\left<\mathcal{A}\left[u\atop v\right],\left[u\atop v\right]\right>_{\mathcal{H}_a}=
\left<\left[ v\atop \left(au_x\right)_x\right],\left[u\atop v\right]\right>_{\mathcal{H}_a}=\sum_{i=1}^2\int_{\Omega_i} \left(au_x\right)_x v\,dx+\sum_{i=1}^2\int_{\Omega_i} av_x u_x\,dx\\
=\lim_{x\nearrow 1}\left[\int_c^x \left(au_s\right)_s v\,ds+\int_c^x av_s u_s\,ds\right]+\lim_{x\searrow 1}\left[\int_x^d \left(au_s\right)_s v\,ds+\int_x^d av_s u_s\,ds\right] \\
= \left[\lim_{x\nearrow 1}a(x)u_x(x)v(x)\right]-
\left[\lim_{x\searrow 1}a(x)u_x(x)v(x)\right]=0,\quad\text{for all }\ \left[u\atop v\right]\in D(\mathcal{A})
\end{multline}
by the transmission conditions. 

Taking into account the equality
\[
\left<\mathcal{A}\left[u\atop v\right],\left[\widetilde{u}\atop \widetilde{v}\right]\right>_{\mathcal{H}_a}= \left<\left[u\atop v\right],\mathcal{A}^\ast\left[\widetilde{u}\atop \widetilde{v}\right]\right>_{\mathcal{H}_a},\quad \left[u\atop v\right], \left[\widetilde{u}\atop \widetilde{v}\right] \in D(\mathcal{A}),
\]
we see that
\begin{multline*}
\left<\mathcal{A}\left[u\atop v\right],\left[\widetilde{u}\atop \widetilde{v}\right]\right>_{\mathcal{H}_a}=
\left<\left[ v\atop \left(au_x\right)_x\right],\left[\widetilde{u}\atop \widetilde{v}\right]\right>_{\mathcal{H}_a}=\sum_{i=1}^2\int_{\Omega_i} \left(au_x\right)_x \widetilde{v}\,dx+\sum_{i=1}^2\int_{\Omega_i} av_x \widetilde{u}_x\,dx\\
=\lim_{x\nearrow 1}\left[\int_c^x \left(au_s\right)_s \widetilde{v}\,ds+\int_c^x av_s \widetilde{u}_s\,ds\right]+\lim_{x\searrow 1}\left[\int_x^d \left(au_s\right)_s \widetilde{v}\,ds+\int_x^d av_s \widetilde{u}_s\,ds\right] \\
=\lim_{x\nearrow 1}\left[-\int_c^x au_s \widetilde{v}_s\,ds-\int_c^x v\left(a \widetilde{u}_s\right)_s\,ds\right]+\lim_{x\searrow 1}\left[-\int_x^d au_s \widetilde{v}_s\,ds-\int_x^d v\left( a \widetilde{u}_s\right)_s\,ds\right]\\
+ \left[\lim_{x\nearrow 1}a(x)u_x(x) \widetilde{v}(x)- \lim_{x\searrow 1}a(x)u_x(x) \widetilde{v}(x)\right]\\
+ \left[\lim_{x\nearrow 1}a(x)\widetilde{u}_x(x) v(x)- \lim_{x\searrow 1}a(x)\widetilde{u}_x(x) v(x)\right]\\
\stackrel{\text{by t.c.}}{=}-\int_{\Omega} \left(a\widetilde{u}_x\right)_x v\,dx-\int_{\Omega} a\widetilde{v}_x u_x\,dx=
\left<\left[ u\atop v\right],\left[-\widetilde{v}\atop -\left(a\widetilde{u}_x\right)_x\right]\right>_{\mathcal{H}_a}.
\end{multline*}
Hence,
$\mathcal{A}^\ast\left[\widetilde{u}\atop \widetilde{v}\right]=\left[-\widetilde{v}\atop -\left(a\widetilde{u}_x\right)_x\right]$,
and arguing as in \eqref{2.1.2}, we deduce that $\mathcal{A}^\ast$ is a dissipative operator. Thus, $\mathcal{A}:D(\mathcal{A})\subset \mathcal{H}_a\rightarrow \mathcal{H}_a$ generates a strongly continuous semi-group of contraction operators.
\end{proof}

 For  further convenience, let us denote this semi-group by $e^{\mathcal{A}t}$.
Then for any $U_0=\left[u_0\atop v_0\right]\in \mathcal{H}_a$, the representation $U(t)= e^{\mathcal{A} t} U_0$ gives the so-called $V^1_a$-mild solution of the Cauchy problem
\begin{equation}
\label{2.4aa}
\left\{
\begin{array}{rcl}
\ds \frac{d}{dt} U(t) &=& \mathcal{A} U(t),\quad t>0,\\
U(0) &=& U_0.
\end{array}
\right.
\end{equation}
When  $U_0\in D(\mathcal{A})$, the solution $U(t)= e^{\mathcal{A} t} U_0$ is classical in the sense that 
$$
U(\cdot)\in C^1([0,\infty); \mathcal{H}_a)\cap C([0,\infty);D(\mathcal{A}))
$$ 
and equation \eqref{2.4aa} holds on $[0,\infty)$.

Thus, in view of the above consideration, we say that, for given $y_0\in V^1_{a,0}(\Omega)$ and $y_1\in L^2(\Omega)$, the function
\[
y\in C^1([0,T]; L^2(\Omega))\cap C([0,T]; V^1_{a,0}(\Omega))
\]
is the $V^1_a$-mild solution of problem
\begin{gather}
\label{2.5.a}
y_{tt}-\left(a(x) y_x\right)_x=0\quad\text{in }\ (0,T)\times\Omega_i,\ i=1,2,\\
\label{2.5.b}
y(t,c)=0,\quad y(t,d)=0,\quad t\in (0,T),\\
\label{2.5.c}
y(0,x)=y_0(x),\quad y_t(0,x)=y_1(x),\quad x\in\Omega,\\
\label{2.5.d}
\text{with the transmission conditions \eqref{1.1.6a}--\eqref{1.1.6b} or \eqref{1.1.7a}--\eqref{1.1.7b},}
\end{gather} 
if $\left[y(t)\atop v(t)\right]=e^{\mathcal{A} t}\left[y_0\atop y_1\right]$ for all $t\in[0,T]$. By the aforementioned regularity result for  $e^{\mathcal{A} t}$,
if
\[
\left[y_0\atop y_1\right]\in W^2_a(\Omega)\times V^1_{a,0}(\Omega),
\]
then $y$ is the $V^1_a$-classical solution of \eqref{2.5.a}--\eqref{2.5.d} meaning that
\[
y\in C^2([0,T]; L^2(\Omega))\cap C^1([0,T]; V^1_{a,0}(\Omega))\cap C([0,T];W^2_a(\Omega))
\]
and the equation  \eqref{2.5.a} is satisfied for all $t\in [0,T]$ and a.e. $x\in \Omega_0$.

The energy of a $V^1_a$-mild solution $y$ of \eqref{2.5.a}--\eqref{2.5.d} is the continuous function defined by 
\[
E_y(t)=\frac{1}{2}\int_{\Omega_0}\left[y^2_t(t,x)+a(x)y^2_x(t,x)\right]\,dx,\quad \forall\,t\ge 0.
\]
\begin{proposition}
\label{Prop 2.2}
Let $a:\overline{\Omega}\rightarrow\mathbb{R}$ be a weight function satisfying properties (i)--(iii), and let $y$ be a $V^1_a$-mild solution of \eqref{2.5.a}--\eqref{2.5.d}. Then
\begin{equation}
\label{2.3}
E_y(t)=E_y(0),\quad\forall\,t\ge 0.
\end{equation}
\end{proposition}
\begin{proof}
Suppose, first, that $y$ is a $V^1_a$-classical solution of \eqref{2.5.a}--\eqref{2.5.d}. Then, multiplying the
equation by $y_t$ and integrating by parts, in view of the transmission conditions \eqref{1.1.6a}--\eqref{1.1.6b} or \eqref{1.1.7a}--\eqref{1.1.7b},  we obtain
\begin{align*}
0=\int_{\Omega_0} &y_t(t,x)y_{tt}(t,x)\,dx-\sum_{i=1}^2 \int_{\Omega_i} y_t(t,x)\left(a(x)y_x(t,x)\right)_x\,dx\\
&=\int_{\Omega_0} \left[y_t(t,x)y_{tt}(t,x)+a(x)y_x(t,x)y_{xt}(t,x)\right]\,dx\\
&\quad-\left[y_t(t,x)a(x)y_x(t,x)\right]_{x=c}^{x=1}
-\left[y_t(t,x)a(x)y_x(t,x)\right]_{x=1}^{x=d}\\
&=\frac{d}{dt} E_y(t)-y_t(t,1)\left(\lim_{x\nearrow 1}\left[a(x)y_x(t,x)\right] - \lim_{x\searrow 1}\left[a(x)y_x(t,x)\right]\right),
\end{align*}
where the last term vanishes because of the transmission conditions. Thus, we conclude that the energy of the $V^1_a$-classical solution $y$ is constant. The same conclusion can be extended to any $V^1_a$-mild solution by  approximation
arguments.
\end{proof}

\section{On Boundary Observability}
\label{Sec_3}
For a given weight function $a:\overline{\Omega}\rightarrow\mathbb{R}$ with properties (i)--(iii), we fix some intermediate space $V^1_{a,0}(\Omega)$ with $H^1_{a,0}(\Omega)\subseteq V^1_{a,0}(\Omega)\subseteq W^1_{a,0}(\Omega)$.
We say that the system \eqref{2.5.a}--\eqref{2.5.d} is boundary $V^1_a$-observable (via the normal derivative at $x = c$ and $x=d$)
in time $T > 0$ if there exists a constant $C_T > 0$ such that for any $y_0\in V^1_{a,0}(\Omega)$ and $y_1\in L^2(\Omega)$ the $V^1_a$-mild solution of \eqref{2.5.a}--\eqref{2.5.d} satisfies the estimate
\begin{equation}
\label{3.1}
\int_0^T y_x^2(t,c)\,dt+\int_0^T y_x^2(t,d)\,dt\ge C_T\, E_y(0).
\end{equation}
Any constant satisfying  \eqref{3.1} is called an $V^1_a$-observability constant for \eqref{2.5.a}--\eqref{2.5.d} in time $T$.
We denoted the supremum of all $V^1_a$-observability constants for \eqref{2.5.a}--\eqref{2.5.d} by $C_T$.

\begin{lemma}
	\label{Lemma 3.1}
For any $V^1_a$-mild solution $y(t,x)$ of \eqref{2.5.a}--\eqref{2.5.d}, we have that $y_x(\cdot,c)\in L^2(0,T)$ and $y_x(\cdot,d)\in L^2(0,T)$ for any $T>0$, and
\begin{gather}
\label{4.1.1}
a(c)\int_0^T y_x^2(t,c)\,dt\le \frac{1}{1-c}\left[\max\{1,C_a^2\}+2T+2T\max\limits_{i=1,2}\left\{\kappa_{i,a}\mu_{i,a} \right\}\right] E_y(0),\\
\label{4.1.1a}
a(d)\int_0^T y_x^2(t,d)\,dt\le \frac{1}{d-1}\left[\max\{1,C_a^2\}+2T+2T\max\limits_{i=1,2}\left\{\kappa_{i,a}\mu_{i,a} \right\}\right] E_y(0),
\end{gather}
where the constant $C_a$ is defined by \eqref{1.0.1dd}. 
Moreover,
\begin{multline}
\label{4.1.2}
(1-c)a(c)\int_0^T y^2_x(t,c)\,dt+(d-1)a(d) \int_0^T y^2_x(t,d)\,dt\\
=
2\int_{\Omega_0}\left[(x-1)y_x(t,x)y_{t}(t,x)\right]_{t=0}^{t=T}\,dx\\
+\int_0^T\int_{\Omega_0}\left(y_t^2(t,x)+\left[1-\frac{(x-1)a_x(x)}{a(x)}\right] a(x)y^2_x(t,x)\right)\,dx\,dt.
\end{multline}
\end{lemma}
\begin{proof}
To begin with, we assume that $\left[y_0\atop y_1\right]\in W^2_a(\Omega)\times V^1_{a,0}(\Omega)$, that is, $y$ given by the formula $\left[y(t)\atop v(t)\right]=e^{\mathcal{A} t}\left[y_0\atop y_1\right]$ is a $V^1_a$-classical solution of the problem \eqref{2.5.a}--\eqref{2.5.d}. Following in many aspects \cite[Lemma 3.2]{Cannarsa}, we multiply  equation \eqref{2.5.a} by $(x-1)y_x$. Integrating over $(0, T)\times \Omega_0$, we obtain 
\begin{align}
\notag
0&=\int_0^T\int_{\Omega_0}(x-1)y_x(t,x)\left(y_{tt}(t,x)-\left(a(x) y_x(t,x)\right)_x\right)\,dx\,dt\\
\notag
&=\int_{\Omega_0}\left[(x-1)y_x(t,x)y_{t}(t,x)\right]_{t=0}^{t=T}\,dx-
\int_0^T\int_{\Omega_0}(x-1)y_{tx}(t,x)y_{t}(t,x)\,dx\,dt\\
\notag
&\quad-\int_0^T\left[(x-1)a(x)y^2_x(t,x)\right]^{x=d}_{x=c}\,dt+\int_0^T\int_{\Omega_0}\Big(a(x)y_x+(x-1)a(x)y_{xx}\Big)y_x\,dx\,dt\\
\notag
&=\int_{\Omega_0}\left[(x-1)y_x(t,x)y_{t}(t,x)\right]_{t=0}^{t=T}\,dx
-\int_0^T\left[(x-1)a(x)y^2_x(t,x)\right]^{x=d}_{x=c}\,dt\\
\label{4.1.3}
&\quad -\int_0^T\int_{\Omega_0}\left((x-1)\left[\frac{y^2_t}{2}\right]_x-(x-1)a(x)\left[\frac{y^2_x}{2}\right]_x-a(x)y_x^2\right)\,dx\,dt
\end{align}
After integration of the last two term, we have
\begin{align}
\notag
\int_0^T\int_{\Omega_0} &(x-1)\left[\frac{y^2_t(t,x)}{2}\right]_x\,dx\,dt=
-\frac{1}{2}\int_0^T\int_{\Omega_0}y^2_t(t,x)\,dx\,dt\\
\notag
&+\frac{1}{2}\int_0^T\left[(x-1)y^2_t(t,x)\right]_{x=c}^{x=1}\,dt+\frac{1}{2}\int_0^T\left[(x-1)y^2_t(t,x)\right]_{x=1}^{x=d}\,dt\\
\label{4.1.4}
&\stackrel{\text{by \eqref{1.1.1a}, \eqref{2.5.b}}}{=}
-\frac{1}{2}\int_0^T\int_{\Omega_0}y^2_t(t,x)\,dx\,dt,\\
\notag
\int_0^T\int_{\Omega_0} &(x-1)a(x)\left[\frac{y^2_x(t,x)}{2}\right]_x\,dx\,dt=
-\frac{1}{2}\int_0^T\int_{\Omega_0}\left[(x-1)a(x)\right]_x y^2_x(t,x)\,dx\,dt\\
\notag
&\quad+\frac{1}{2}\int_0^T \left[(x-1)a(x)y^2_x(t,x)\right]_{x=c}^{x=1}\,dt
+\frac{1}{2}\int_0^T \left[(x-1)a(x)y^2_x(t,x)\right]_{x=1}^{x=d}\,dt\\
\notag
&\stackrel{\text{by \eqref{1.1.1e}, \eqref{2.5.b}}}{=}
\frac{(1-c)a(c)}{2} \int_0^T y^2_x(t,c)\,dt+\frac{(d-1)a(d)}{2} \int_0^T y^2_x(t,d)\,dt\\
\label{4.1.5}
&\quad-\frac{1}{2}\int_0^T\int_{\Omega_0}\left[(x-1)a(x)\right]_x y^2_x(t,x)\,dx\,dt.
\end{align}
As a result, the identity \eqref{4.1.2} follows by inserting \eqref{4.1.4} and 
\eqref{4.1.5} into \eqref{4.1.3}. To deduce the estimate \eqref{4.1.1}, it is enough to notice that 
\begin{align}
\notag
&\left|\int_{\Omega_0}(x-1)y_x(t,x)y_{t}(t,x)\,dx\right|\le \frac{1}{2}\int_{\Omega_0}\left[y^2_{t}(t,x)+\frac{(x-1)^2}{a(x)} a(x)y^2_x(t,x)\right]\,dx\\
\notag
&\stackrel{\text{by \eqref{1.0.1d}--\eqref{1.0.1e}}}{\le}E_y(0)\max\left\{1,\frac{ (1-c)^{\mu_{1,a}}}{\min\limits_{x\in[c,x_1^\ast]} a(x)}, \frac{(1-x_1^\ast)^{\mu_{1,a}}}{a(x_1^\ast)},\frac{ (d-1)^{\mu_{2,a}}}{\min\limits_{x\in[x_2^\ast,d]} a(x)}, \frac{(x_2^\ast-1)^{\mu_{2,a}}}{a(x_2^\ast)}\right\}\\
\label{4.1.5.1}
&\stackrel{\text{by \eqref{1.0.1dd}}}{=} \frac{\max\{4,C_a^2\}}{4} E_y(0),\\
\notag
&\left[1-\frac{(x-1)a_x(x)}{a(x)}\right]\stackrel{\text{by \eqref{1.0.A1}, \eqref{1.0.A2}}}{\le} 1+\max\left\{\kappa_{1,a}\mu_{1,a},\kappa_{2,a}\mu_{2,a}\right\}\quad\text{in }\ \Omega,
\end{align}
and  the energy $E_y(t)$ is constant. 

In order to extend relations \eqref{4.1.1} and \eqref{4.1.2} to the $V^1_a$-mild solution associated with the initial data $y_0\in V^1_{a,0}(\Omega)$ and $y_1\in L^2(\Omega)$, it suffices to approximate such data by $\left[y^k_0\atop y^k_1\right]\in W^2_a(\Omega)\times V^1_{a,0}(\Omega)$ and use estimate \eqref{4.1.1} to show that the normal derivatives of the corresponding classical solutions give a Cauchy sequence in $L^2(0,T)$.
\end{proof}

\begin{lemma}
	For any $V^1_a$-mild solution $y(t,x)$ of \eqref{2.5.a}--\eqref{2.5.d} we have: for each $T>0$,
	\begin{equation}
	\label{4.1.6}
	\int_0^T\int_{\Omega_0} \Big[a(x) y^2_x(t,x)-y_{t}^2(t,x)\Big]\,dx\,dt+\int_{\Omega_0} \left[y(t,x) y_{t}(t,x)\right]_{t=0}^{t=T}\,dx=0.
	\end{equation}
\end{lemma}
\begin{proof}
	Let $y$ be a $V^1_a$-classical solution of \eqref{2.5.a}--\eqref{2.5.d}. Then, multiplying equation \eqref{2.5.a} by $y$ and integrating over $(0,T)\times\Omega_0$, we obtain
	\begin{align*}
	0&=\int_0^T\int_{\Omega_0} y(t,x)\left[y_{tt}(t,x)-\left(a(x) y_x(t,x)\right)_x\right]\,dx\,dt 
	=\left[\int_{\Omega_0} y(t,x) y_{t}(t,x)\,dx\right]_{t=0}^{t=T}\\
	&\quad-\int_0^T\int_{\Omega_0} y_{t}^2(t,x)\,dx\,dt
	-\int_0^T  \left[a(x) y_x(t,x) y(t,x)\right]_{x=c}^{x=1}\,dt\\
	&\quad-\int_0^T  \left[a(x) y_x(t,x) y(t,x)\right]_{x=1}^{x=d}\,dt
	+\int_0^T\int_{\Omega_0} a(x) y^2_x(t,x)\,dx\,dt.
	\end{align*}
Since
\begin{multline*}
\int_0^T  \left[a(x) y_x(t,x) y(t,x)\right]_{x=c}^{x=1}\,dt
+\int_0^T  \left[a(x) y_x(t,x) y(t,x)\right]_{x=1}^{x=d}\,dt\\
\stackrel{\text{by \eqref{2.5.b}}}{=} \int_0^T \left[\lim_{x\nearrow 1}a(x) y_x(t,x) y(t,x)-\lim_{x\searrow 1}a(x) y_x(t,x) y(t,x)\right]\,dt =0
\end{multline*}
by the transmission conditions \eqref{1.1.6a}--\eqref{1.1.6b} or \eqref{1.1.7a}--\eqref{1.1.7b}, the announced equality \eqref{4.1.6} follows from the above identity. Then the approximation arguments allow to
extend this conclusion to $V^1_a$-mild solutions.
\end{proof}

\begin{theorem}
	\label{Th 4.1.1}
	Let $a:\overline{\Omega}\rightarrow\mathbb{R}$ be a weight function satisfying properties (i)--(iii). In addition, we assume that
	\begin{gather}
	\label{4.0.1}
	\frac{d \ln a(x)}{dx}  \ge \frac{d \ln (1-x)^{\mu_{1,a}}}{dx},\quad\forall\,x\in [c,x_1^\ast],\\
	\label{4.0.2}
	\frac{d \ln a(x)}{dx}  \le \frac{d \ln (x-1)^{\mu_{2,a}}}{dx},\quad\forall\,x\in [x_2^\ast,d].
	\end{gather}
	Let $y$ be a $V^1_a$-mild solution of \eqref{2.5.a}--\eqref{2.5.d}. Then, for every $T>0$, the estimate
	\begin{multline}
	\label{4.1.7}
	(1-c)a(c) \int_0^T y^2_x(t,c)\,dt+(d-1)a(d) \int_0^T y^2_x(t,d)\,dt\\
	\ge \left[
	\left(2-\max\{\mu_{1,a},\mu_{2,a}\}\right)T-
	\max\{4,C_a^2\}-\min\{D_a,C_a\}\max\left\{\mu_{1,a},\mu_{2,a}\right\}
	\right]E_y(0)
	\end{multline}
	holds true with $C_a$ and $D_a$ are given by relations \eqref{1.0.1dd} and \eqref{2.4ab}, respectively.
\end{theorem}
\begin{proof}
	Since the case of mild solutions can be recovered by approximation arguments, we restrict ourself by assumptions that $y$ is a $V^1_a$-classical solution of the problem \eqref{2.5.a}--\eqref{2.5.d}. Then adding to the right hand side of \eqref{4.1.2} the left side of \eqref{4.1.6} multiplied by
	$$
	B_a:=\frac{1}{2}\max\left\{\mu_{1,a},\mu_{2,a}\right\},
	$$ 
	we obtain 
	\begin{multline*}
	(1-c)a(c) \int_0^T y^2_x(t,c)\,dt+(d-1)a(d) \int_0^T y^2_x(t,d)\,dt\\
	=2\int_{\Omega_0}\left[(x-1)y_x(t,x)y_{t}(t,x)\right]_{t=0}^{t=T}\,dx+
	B_a \int_{\Omega_0} \left[y(t,x) y_{t}(t,x)\right]_{t=0}^{t=T}\,dx\\
	+\int_0^T\int_{\Omega_0}\left(1-B_a\right)y_t^2(t,x)\,dx\,dt\\
	+\int_0^T\int_{\Omega_0}\left(\left[1+B_a-\frac{(x-1)a_x(x)}{a(x)}\right] a(x)y^2_x(t,x)\right)\,dx\,dt
	=I_1+I_2+I_3+I_4.
	\end{multline*}
	Since
	\begin{gather*}
	-\frac{(x-1)a_x(x)}{a(x)}\ge -\frac{|x-1||a_x(x)|}{a(x)}\ge -\max\{\mu_{1,a},\mu_{2,a}\}\quad\text{in }\ [x_1^\ast,x_2^\ast],\\
	-\frac{(x-1)a_x(x)}{a(x)}\stackrel{\text{by \eqref{4.0.1}}}{\ge}-\mu_{1,a}\quad\forall\,x\in [c,x_1^\ast],\\
	-\frac{(x-1)a_x(x)}{a(x)}\stackrel{\text{by \eqref{4.0.2}}}{\ge}-\mu_{2,a}\quad\forall\,x\in [x_2^\ast,d],
	\end{gather*}
	it follows that
	\[
	I_3+I_4\ge \left(2-\max\{\mu_{1,a},\mu_{2,a}\}\right)\int_0^T E_y(0)\,dt=
	\left(2-\max\{\mu_{1,a},\mu_{2,a}\}\right)T
	 E_y(0).
	\]
	Taking into account that
	\[
	I_1=2\int_{\Omega_0}\left[(x-1)y_x(t,x)y_{t}(t,x)\right]_{t=0}^{t=T}\,dx\stackrel{\text{by \eqref{4.1.5.1}}}{\ge} -\max\{4,C_a^2\} E_y(0),
	\]
	and 
	\begin{align*}
	\left|\int_{\Omega_0} y(t,x) y_{t}(t,x)\,dx\right|&\le
	\frac{1}{2}\int_{\Omega_0} \left(\frac{1}{\min\{D_a,C_a\}} y^2(t,x)+\min\{D_a,C_a\} y_t^2(t,x)\right)\,dx
	\\
	&\le \min\{D_a,C_a\} E_y(0),
	\end{align*}
	where $\min\{D_a,C_a\}$ is Poincar\'{e}'s constant in \eqref{1.0.1ddd}, we see that
	\[
	I_2\ge -2B_a \min\{D_a,C_a\} E_y(0)=-\min\{D_a,C_a\}\max\left\{\mu_{1,a},\mu_{2,a}\right\}E_y(0).
	\]
	Thus, the announced estimate \eqref{4.1.7} is proven.
\end{proof}

Due to Theorem~\ref{Th 4.1.1}, the observability constant $C_T$ (see inequality \eqref{3.1}) for the problem  \eqref{2.5.a}--\eqref{2.5.d} in time $T$ can be derived from \eqref{4.1.7}. Namely,
\begin{multline*}
C_T=\frac{1}{\max\{(1-c)a(c),(d-1)a(d)\}} \\
\times\left[
\left(2-\max\{\mu_{1,a},\mu_{2,a}\}\right)T-
\max\{4,C_a^2\}-\min\{D_a,C_a\}\max\left\{\mu_{1,a},\mu_{2,a}\right\}
\right].
\end{multline*}

As for the minimal time $T_a>0$ when the system \eqref{2.5.a}--\eqref{2.5.d} becomes boundary $V^1_a$-observable in time $T>T_a$, it can be defined as follows
\begin{equation}
\label{4.17}
T_a:=\frac{1}{\left(2-\max\{\mu_{1,a},\mu_{2,a}\}\right)}\left[\max\{4,C_a^2\}+\min\{D_a,C_a\}\max\left\{\mu_{1,a},\mu_{2,a}\right\}\right].
\end{equation}

\begin{example}
	\label{Ex 4.1}
	Setting $c=0$, $d=2$, and $a(x)=|x-1|^p$, we see that the initial assumptions (i)--(iii) holds true  with $p\in (0,2)$, and
	\begin{equation*}
	x_1^\ast=0,\quad x_2^\ast=2,\quad\text{and}\quad \mu_{1,a}=\mu_{2,a}=p.
	\end{equation*}
	Then we deduce from \eqref{2.4ab} and \eqref{1.0.1dd} that
	$C_a^2=4$ and $D_a^2=\frac{1}{2-p}$. 
	Since 
	$$
	\min\{D_a,C_a\}=\min\left\{\sqrt{\frac{1}{2-p}},2\right\}=\left\{
	\begin{array}{ll}
	\sqrt{\frac{1}{2-p}}, &\ \text{if }\ p\in (0,\frac{7}{4}),\\
	2, &\  \text{if }\ p\in [\frac{7}{4}, 2),
	\end{array}
	\right.,
	$$
	it follows from \eqref{4.17} that 
	\begin{align*}
	T_a&=\ds\frac{p+4\sqrt{2-p}}{(2-p)\sqrt{2-p}}, \quad \text{for all }\ p\in \left(0,\frac{7}{4}\right),\\
	T_a&=\ds 2\left(\frac{2+p}{2-p}\right), \quad  \text{for all  }\ p\in \left[\frac{7}{4}, 2\right).	
	\end{align*}
	It is worth to notice here that $T_a\searrow 2$ as $p\searrow 0$. In this case, the damage effect at the middle point $x=1$ disappears and $T_a$ coincides with the classical observability time for the wave equations on the two connected planar strings (see \cite[Section 4.4]{Dager}). At the same time we aarive at the blow up of the observability time if $p\nearrow 2$.
\end{example}

\section{On Boundary Null Controllability}
\label{Sec 4}
In this section the problem of boundary controllability of the degenerate wave equation is studied. The control is assumed to act at the boundary points $x=c$ and $x=d$ through the Dirichlet conditions. So, we consider the following degenerate control system
\begin{gather}
\label{5.1.a}
y_{tt}-\left(a(x) y_x\right)_x=0\quad\text{in }\ (0,+\infty)\times\Omega_i,\ i=1,2,\\
\label{5.1.b}
y(t,c)=f_c(t),\quad y(t,d)=f_d(t),\quad t\in (0,+\infty),\\
\label{5.1.c}
y(0,x)=y_0(x),\quad y_t(0,x)=y_1(x),\quad x\in\Omega,\\
\label{5.1.d}
\text{with the transmission conditions \eqref{1.1.6a}--\eqref{1.1.6b} or \eqref{1.1.7a}--\eqref{1.1.7b},}
\end{gather} 
where $f_c, f_d\in L^2(0,T)$ are the controls. 

By analogy with the previous section, for a given weight function $a:\overline{\Omega}\rightarrow\mathbb{R}$ with properties (i)--(iii), we fix some intermediate space $V^1_{a,0}(\Omega)$ with $H^1_{a,0}(\Omega)\subseteq V^1_{a,0}(\Omega)\subseteq W^1_{a,0}(\Omega)$.
Let $V^{-1}_a(\Omega)$ be the dual space to $V^1_{a,0}(\Omega)$ with respect to the pivot space $L^2(\Omega)$. 
In order to make a precise definition of the solution to the boundary value problem \eqref{5.1.a}--\eqref{5.1.d}, where $f_c, f_d\in L^2(0,T)$ are the controls, and indicate its characteristic properties, we notice that by Theorem~\ref{Th 2.0.1} (see also \eqref{1.0.0}) the operator $A_a:D(A_a)\subset L^2(\Omega)\rightarrow L^2(\Omega)$, where $A_a(y)=-(ay_x)_x$ and $D(A_a)=\left\{ y\in V^1_{a,0}(\Omega)\ :\ a y_x\in W^{1,2}(\Omega)\right\}$, is an isomorphism from $V^1_{a,0}(\Omega)$ onto $V^{-1}_a(\Omega)$. In particular, $V^{-1}_a(\Omega)=A_a\left(V^1_{a,0}(\Omega)\right)$.
\begin{definition}
	\label{Def 5.1}
	System \eqref{5.1.a}--\eqref{5.1.d} is boundary null controllable in time $T>0$ if, for every initial data $y_0\in L^2(\Omega)$, and $y_1\in V^{-1}_a(\Omega)$, the set of reachable states $(y(T), y_t(T))$, where $y$ is a solution of \eqref{5.1.a}--\eqref{5.1.d} with $f_c, f_d\in L^2(0,T)$, contains the element $(0,0)$.
\end{definition}
\begin{definition}
	\label{Def 5.1a}
	System \eqref{5.1.a}--\eqref{5.1.d} is boundary exactly controllable in time $T>0$ if, for every initial data $y_0\in L^2(\Omega)$, and $y_1\in V^{-1}_a(\Omega)$, the set of reachable states $(y(T), y_t(T))$, coincides with $L^2(\Omega)\times V^{-1}_a(\Omega)$.
\end{definition}
\begin{remark}
	Arguing as in Proposition~2.2.1 in \cite{Zuazua}, and utilizing the linearity and reversibility properties of system \eqref{5.1.a}--\eqref{5.1.d}, it can be shown that this system is exactly controllable through the boundary Dirichlet conditions at $x=c$ and $x=d$ if and only if it is null controllable.
\end{remark}

Following the standard approach and utilizing the transmission conditions, we define the solution of  controlled system \eqref{5.1.a}--\eqref{5.1.d} by transposition. 
\begin{definition}
	\label{Def 5.2}
	Let $f_c, f_d\in L^2(0,T)$, $y_0\in L^2(\Omega)$, and $y_1\in V^{-1}_a(\Omega)$ be given distributions. We say that $y$ is a $V^1_a$-solution by transposition of the problem \eqref{5.1.a}--\eqref{5.1.d} if
	\[
	y\in C^1\left([0,\infty); V^{-1}_a(\Omega)]\right)\cap C\left([0,\infty); L^2(\Omega)\right)
	\]
	satisfies for all $T>0$ and all $w^0_T\in V^1_{a,0}(\Omega)$ and $w^1_T\in L^2(\Omega)$ the following equality
	\begin{align}
	\notag
	\left<y_t(T),w^0_T\right>_{V^{-1}_a(\Omega);V^1_{a,0}(\Omega)} &- \int_{\Omega} y(T) w^1_T\,dx \\
	\notag
	&=\left<y_1,w(0)\right>_{V^{-1}_a(\Omega);V^1_{a,0}(\Omega)}-\int_{\Omega} y_0 w_t(0)\,dx\\
	&\quad-a(d)\int_0^T f_d(t) w_x(t,d)\,dt+a(c)\int_0^T f_c(t) w_x(t,c)\,dt,
	\label{5.2}
	\end{align}
	where $w$ is the solution of the backward homogeneous equation
	\begin{equation}
	\label{5.3}
	w_{tt}-(a(x) w_x)_x=0\quad\text{in }\ (0,+\infty)\times\Omega_i,\ i=1,2
	\end{equation}
	with the final conditions
	\begin{equation}
	\label{5.4}
	w(T)=w^0_T,\quad w_t(T)=w^1_T\quad\text{ in }\ \Omega,
	\end{equation}
	the boundary conditions
	\begin{equation}
	\label{5.5}
	w(t,c)=0,\quad
	w(t,d)=0\quad\text{on }\ (0,T),
	\end{equation}
	and  the transmission conditions:
	\begin{enumerate}
		\item[(I)] For the case $1/a\in L^1(\Omega)$
		\begin{gather}
		\label{5.6a}
		\lim_{x\nearrow 1} w(t)=\lim_{x\searrow 1} w(t),\quad 0<t<T,\\
		\label{5.6b}
		\lim_{x\nearrow 1} a w_x(t)=\lim_{x\searrow 1} a w_x(t),\quad 0<t<T;
		\end{gather}
		
		\item[(II)] For the case $1/a\not\in L^1(\Omega)$
		\begin{gather}
		\label{5.7a}
		\lim_{x\nearrow 1} a\varphi_x w(t)=0=\lim_{x\searrow 1} a \varphi_x w(t),\ \forall\,\varphi\in V^1_{a,0}(\Omega),\  0<t<T,\\
		\label{5.7b}
		\lim_{x\nearrow 1} a w_x(t)=0=\lim_{x\searrow 1} a w_x(t),\quad 0<t<T.
		\end{gather}
	\end{enumerate}
\end{definition}

Following the results of Section~\ref{Sec 2} and making the change of variable $u(t,x)=w(T-t,x)$, we see that the backward problem \eqref{5.3}--\eqref{5.7b} admits a unique $V_a^1$-mild solution $w\in C^1\left([0,T]; L^2(\Omega)]\right)\cap C\left([0,T]; V^1_{a,0}(\Omega)\right)$ for each $T>0$. Moreover, arguing as in Lemma~\ref{Lemma 3.1}, it can be shown that there exists a constant $C>0$ such that
\begin{equation}
\label{5.5.1}
\int_0^T w_x^2(t,c)\,dt+\int_0^T w_x^2(t,d)\,dt\ge C\, E_w(T),
\end{equation}
where 
\[
E_w(t)=\frac{1}{2}\int_{\Omega_0}\left[w^2_t(t,x)+a(x)w^2_x(t,x)\right]\,dx=E_w(T),\quad \forall\,t\in [0,T],
\]
is the energy of a $V_a^1$-mild solution $w$ and it is conserved through time. Since 
\begin{equation}
\label{5.5.2}
E_w(T)=\frac{1}{2}\left[\|w^1_T\|^2_{L^2(\Omega)}+\|w^0_T\|^2_{V^1_{a,0}(\Omega)}\right],
\end{equation}
it follows that a $V_a^1$-mild solution $w$ of \eqref{5.3}--\eqref{5.7b} depends continuously on the data  $(w^0_T,w^1_T)\in V^1_{a,0}(\Omega)\times L^2(\Omega)$, and, therefore, the right hand side of \eqref{5.2} defines a continuous linear form with respect to $(w^0_T,w^1_T)\in V^1_{a,0}(\Omega)\times L^2(\Omega)$ $T>0$. Thus, a $V_a^1$-solution $y$ by transposition of \eqref{5.1.a}--\eqref{5.1.d} is unique in $C^1\left([0,\infty); V^{-1}_a(\Omega)]\right)\cap C\left([0,\infty); L^2(\Omega)\right)$. The following theorem is a consequence of the classical results of existence and uniquencess of solutions of nonhomogeneous evolution equations. Full details can be found in \cite{LionsMag} and \cite{Zuazua90}.
\begin{theorem}
	For any $f_c, f_d\in L^2(0,T)$ and $(y_0,y_1)\in L^2(\Omega)\times V^{-1}_a(\Omega)$ transmission problem \eqref{5.1.a}--\eqref{5.1.d}  has a unique $V_a^1$-solution defined by transposition
	\[
	\left(y, y_t\right)\in C\left([0,T]; L^2(\Omega)\times V^{-1}_a(\Omega)\right).
	\]
	Moreover, the map $\left(y_0,y_1,f_c, f_d\right)\mapsto \left\{y,y_t\right\}$ is linear and there exists a constant $C(T)>0$ such that
	\begin{multline*}
	\|y\|_{L^\infty(0,T;L^2(\Omega))}+ \|y_t\|_{L^\infty(0,T;V^{-1}_a(\Omega))}\\
	\le C(T)\left[\|y_0\|_{L^2(\Omega)}+\|y_1\|_{V^{-1}_a(\Omega)}+\|f_c\|_{L^2(0,T)}+\|f_d\|_{L^2(0,T)}\right].
	\end{multline*}
\end{theorem}

We are now in a position to prove the main result of this section.
\begin{theorem}
	\label{Th 5.5}
	Let $a:\overline{\Omega}\rightarrow\mathbb{R}$ be a weight function satisfying properties (i)--(iii)  and \eqref{4.0.1}--\eqref{4.0.2}. Let $T_a$ be a value defined as in \eqref{4.17}. Then, for every $T>T_a$ and for any $(y_0,y_1)\in L^2(\Omega)\times V^{-1}_a(\Omega)$, there exists a pair of controls $f_c, f_d\in L^2(0,T)$ such that the $V^1_a$-solution of \eqref{5.1.a}--\eqref{5.1.d} (in the sense of transposition) satisfies condition $\left(y(T),y_t(T)\right)\equiv (0,0)$, i.e. the system \eqref{5.1.a}--\eqref{5.1.d} is boundary null controllable in time $T>T_a$.
\end{theorem}
\begin{proof}
	Let 
	$	\left[y_0\atop y_1\right]\in L^2(\Omega)\times V^{-1}_a(\Omega),\quad
	\left[w^0_T\atop w^1_T\right]$, $	\left[\widehat{w}^0_T\atop \widehat{w}^1_T\right]\in V^1_{a,0}(\Omega)\times L^2(\Omega)$
	be arbitrary pairs. Let $w$ and $\widehat{w}$ be $V^1_a$-mild solutions of the backward problem \eqref{5.3}--\eqref{5.7b} with final conditions $\left[w^0_T\atop w^1_T\right]$  and  	$\left[\widehat{w}^0_T\atop \widehat{w}^1_T\right]$, respectively. Let us define the bilinear form $\Lambda$ on $V^1_{a,0}(\Omega)\times L^2(\Omega)$ as follows
	\begin{gather*}
	\Lambda\left(\left[w^0_T\atop w^1_T\right], \left[\widehat{w}^0_T\atop \widehat{w}^1_T\right]\right):=a(c)\int_0^T w_x(t,c) \widehat{w}_x(t,c)\,dt+a(d)\int_0^T w_x(t,d) \widehat{w}_x(t,d)\,dt,\\ \forall\, 
	\left[w^0_T\atop w^1_T\right],\ 	\left[\widehat{w}^0_T\atop \widehat{w}^1_T\right]\in V^1_{a,0}(\Omega)\times L^2(\Omega).
	\end{gather*}
	Then, in view of estimate \eqref{5.5.1} and representation \eqref{5.5.2}, we deduce that the bilinear form $\Lambda:\left[V^1_{a,0}(\Omega)\times L^2(\Omega)\right]^2\to\mathbb{R}$ is continuous. Moreover, due to Theorem~\ref{Th 4.1.1} and observability inequality \eqref{4.1.7}, this form is coercive on  $V^1_{a,0}(\Omega)\times L^2(\Omega)$ provided $T>T_a$. Thus, by the Lax-Milgram Lemma, variational problem
	\[
	\Lambda\left(\left[w^0_T\atop w^1_T\right], \left[\widehat{w}^0_T\atop \widehat{w}^1_T\right]\right)=\left<y_1,\widehat{w}(0)\right>_{V^{-1}_a(\Omega);V^1_{a,0}(\Omega)}-\int_{\Omega} y_0 \widehat{w}_t(0)\,dx,\ \forall\,	\left[\widehat{w}^0_T\atop \widehat{w}^1_T\right]\in V^1_{a,0}(\Omega)\times L^2(\Omega)
	\]
	has a unique solution $\left[w^0_T\atop w^1_T\right]\in V^1_{a,0}(\Omega)\times L^2(\Omega)$. Then setting $f_c=-w_x(t,c)$, $f_d=w_x(t,d)$, and $T>T_a$, where $w\in C^1\left([0,T]; L^2(\Omega)]\right)\cap C\left([0,T]; V^1_{a,0}(\Omega)\right)$ is a $V_a^1$-mild solution of the backward problem \eqref{5.3}--\eqref{5.7b} with $\left[w^0_T\atop w^1_T\right]$ as the final data, we see that
	\begin{multline}
	\label{5.8}
	a(d)\int_0^T f_d(t) \widehat{w}_x(t,d)\,dt-a(c)\int_0^T f_c(t) \widehat{w}_x(t,c)\,dt\\
	=a(c)\int_0^T w_x(t,c) \widehat{w}_x(t,c)\,dt+a(d)\int_0^T w_x(t,d) \widehat{w}_x(t,d)\,dt
	=\Lambda\left(\left[w^0_T\atop w^1_T\right], \left[\widehat{w}^0_T\atop \widehat{w}^1_T\right]\right)\\
	=\left<y_1,\widehat{w}(0)\right>_{V^{-1}_a(\Omega);V^1_{a,0}(\Omega)}-\int_{\Omega} y_0 \widehat{w}_t(0)\,dx,\quad \forall\,	\left[\widehat{w}^0_T\atop \widehat{w}^1_T\right]\in V^1_{a,0}(\Omega)\times L^2(\Omega).
	\end{multline} 
	On the other hand, if $y$ is the $V_a^1$-solution by transposition of the problem  \eqref{5.1.a}--\eqref{5.1.d}, then equality \eqref{5.2} implies that, for all $\left[\widehat{w}^0_T\atop \widehat{w}^1_T\right]\in V^1_{a,0}(\Omega)\times L^2(\Omega)$, we have
	\begin{multline}
	\label{5.9}
	a(d)\int_0^T f_d(t) \widehat{w}_x(t,d)\,dt-a(c)\int_0^T f_c(t) \widehat{w}_x(t,c)\,dt=\left<y_1,\widehat{w}(0)\right>_{V^{-1}_a(\Omega);V^1_{a,0}(\Omega)}\\-\int_{\Omega} y_0 \widehat{w}_t(0)\,dx
	-\left<y_t(T),\widehat{w}^0_T\right>_{V^{-1}_a(\Omega);V^1_{a,0}(\Omega)} + \int_{\Omega} y(T) \widehat{w}^1_T\,dx.
	\end{multline}
	Comparing the last relations \eqref{5.8}--\eqref{5.9}, we obtain
	\[
	-\left<y_t(T),\widehat{w}^0_T\right>_{V^{-1}_a(\Omega);V^1_{a,0}(\Omega)} + \int_{\Omega} y(T) \widehat{w}^1_T\,dx=0,\quad \forall\,	\left[\widehat{w}^0_T\atop \widehat{w}^1_T\right]\in V^1_{a,0}(\Omega)\times L^2(\Omega).
	\]
	From this we finally deduce that $\left(y(T),y_t(T)\right)\equiv (0,0)$, i.e. the system \eqref{5.1.a}--\eqref{5.1.d} is boundary null controllable in time $T>T_a$.
\end{proof}

As an obvious consequence of this result, we can give the following its generalization to the case when only one boundary point $x=d$ is controlled. So, formally setting $f_c(t)\equiv 0$ in \eqref{5.1.b}, we consider the boundary null controllability problem for the system \eqref{5.1.a}--\eqref{5.1.d} with only one control $f_d$. Then arguing as in the proof of Theorem~\ref{Th 5.5}, we arrive at the following conclusion.
\begin{theorem}
	\label{Th 5.6}
	Let $a:\overline{\Omega}\rightarrow\mathbb{R}$ be a weight function  such that it satisfies properties (i)--(iii) and \eqref{4.0.1}--\eqref{4.0.2}. Assume that the bilinear form 
	\begin{gather*}
	\Lambda\left(\left[w^0_T\atop w^1_T\right], \left[\widehat{w}^0_T\atop \widehat{w}^1_T\right]\right):=a(d)\int_0^T w_x(t,d) \widehat{w}_x(t,d)\,dt,\quad \forall\, 
	\left[w^0_T\atop w^1_T\right],\ 	\left[\widehat{w}^0_T\atop \widehat{w}^1_T\right]\in V^1_{a,0}(\Omega)\times L^2(\Omega)
	\end{gather*}
	is coercive on  $H^1_{a,0}(\Omega)\times L^2(\Omega)$, i.e., there exists a positive constant $C_1>0$ such that
	\[
	\Lambda\left(\left[w^0_T\atop w^1_T\right], \left[w^0_T\atop w^1_T\right]\right)\ge C_1\left[\|w^0_T\|^2_{V^1_{a,0}(\Omega)}+\|w^1_T\|^2_{L^2(\Omega)}\right],\ \forall\, 
	\left[w^0_T\atop w^1_T\right]\in V^1_{a,0}(\Omega)\times L^2(\Omega).
	\]
	Then there exists a control $f_d\in L^2(0,T)$ such that the system \eqref{5.1.a}--\eqref{5.1.d} with $f_c(t)\equiv 0$ is one side boundary null controllable in time $T$. 
\end{theorem}

\section*{Acknowledgment}
The first and third author acknowledge financial support by the DFG under contract Le595/31-1: Sustainable Optimal Controls for Nonlinear Partial Differential Equations with Applications. The first author was also partially supported by National Research Foundation of Ukraine (Grant No.~2020.02/0066).

\end{document}